\documentclass{amsart}

\usepackage{amscd,amssymb,amsmath,amsfonts,bbm,calligra,xspace,amsthm}
\usepackage[T1]{fontenc}
\usepackage{lmodern}
\usepackage{mathtools}
\usepackage{enumitem}
\usepackage{multicol}
\usepackage[all]{xy}
\usepackage{mathrsfs}
\usepackage{footmisc}
\usepackage{xcolor}
\definecolor{darkblue}{rgb}{0,0,0.8}
\definecolor{darkgreen}{rgb}{0,0.4,0}
\usepackage[colorlinks=true,linkcolor=darkblue, citecolor=darkgreen,  urlcolor=darkgreen, unicode,pdfborder={0 0 0},final]{hyperref}

\usepackage{mathrsfs,mathabx}
\usepackage{tikz-cd}
\usepackage{indentfirst}
\usepackage{verbatim}
\usepackage{geometry}
\usepackage{ulem}

\newtheorem{thm}{Theorem}[section]
\newtheorem{prop}[thm]{Proposition}

\newtheorem{lem}[thm]{Lemma}

\theoremstyle{definition}
\newtheorem{Def}[thm]{Definition}
\newtheorem{quest}[thm]{Question}

\theoremstyle{remark}
\newtheorem{rem}[thm]{Remark}

\numberwithin{equation}{section}

\newcommand{\Id}{\mathrm{Id}}

\newcommand{\cl}{\mathrm{cl}}

\newcommand{\N}{\mathrm{N}}

\newcommand{\sm}{\mathrm{sm}}

\newcommand{\Spec}{\mathrm{Spec}}

\newcommand{\Pic}{\mathrm{Pic}}

\newcommand{\MU}{\mathrm{MU}}

\newcommand{\CH}{\mathrm{CH}}
\newcommand{\Gr}{\mathrm{Gr}}

\newcommand{\op}{\mathrm{op}}
\newcommand{\fl}{\mathrm{fl}}
\newcommand{\Ch}{\mathrm{Ch}}
\newcommand{\rmsm}{\mathrm{sm}}
\newcommand{\rmD}{\mathrm{D}}
\newcommand{\rmVAD}{\mathrm{VAD}}
\def\alg{\mathrm{alg}}

\newcommand{\Sm}{\mathrm{\mathbf{Sm}}}
\newcommand{\R}{\mathrm{\mathbf{R}}}
\newcommand{\pt}{\mathrm{pt}}

\newcommand{\bbC}{\mathbb{C}}

\newcommand{\bbL}{\mathbb{L}}

\newcommand{\bbP}{\mathbb{P}}

\newcommand{\bbZ}{\mathbb{Z}}

\newcommand{\cO}{\mathcal{O}}

\newcommand{\wt}{\widetilde}

\newcommand{\ra}{\rightarrow}

\newcommand{\hra}{\hookrightarrow}

\newcommand{\thra}{\twoheadrightarrow}

\newcommand{\xra}{\xrightarrow}

\newcommand{\tbf}{\textbf}

\begin{document}

\title[Smoothing low-dimensional cycles in Algebraic Cobordism]{Smoothing low-dimensional cycles in Algebraic Cobordism}

\author{Chuhao Huang}
\address{D\'epartement de math\'ematiques et applications, \'Ecole normale sup\'erieure, CNRS,
45 rue d'Ulm, 75230 Paris Cedex 05, France}
\email{chuhao.huang@ens.fr}

\renewcommand{\abstractname}{Abstract}

\begin{abstract}
    We show that every cycle in the degree $d$ algebraic cobordism group $\Omega_d(X)$ of a smooth projective variety $X$ over a field of characteristic $0$ is smoothable when $2d<\dim(X)$, that is, it can be written as a linear combination of cycles represented by smooth closed subvarieties of $X$. This generalizes a result of Koll\'ar and Voisin from Chow groups to algebraic cobordism groups.
\end{abstract}

\maketitle

\section{Introduction}

Throughout this paper, we work over a fixed field $k$ of characteristic $0$.

\subsection{Smoothability of algebraic cycles}

The question of smoothability of algebraic cycles was first asked over $k=\bbC$ and in the context of singular cohomology by Borel and Haefliger in \cite[Section 6.17]{Borel1961}. The question can be stated as follows. Let $X$ be a smooth complex projective variety of dimension~$n$. We define $H^*(X,\bbZ)_{\alg}$ to be the subgroup of $H^*(X,\bbZ)$ generated by fundamental classes of subvarieties of $X$, and define $H^*(X,\bbZ)_{\sm}$ similarly with the additional requirement that the subvarieties be smooth. Classes in $H^*(X,\bbZ)_{\sm}$ are called smoothable. The question is whether we have~$H^*(X,\bbZ)_{\alg}=H^*(X,\bbZ)_{\sm}$ or not.

It was soon realized that this equality does not hold in general. The first negative result was obtained by Hartshorne, Rees and Thomas in \cite{Hartshorne1974}. They show that the second Chern class~$c_2(T)$ of the tautological bundle $T$ over the complex Grassmannian $G(3,n)$ is never smoothable for $n\geq 6$. Benoist and Debarre prove in \cite{Debarre1995,BenoistDebarre2023} that for a very general complex Jacobian $(X,\theta)$ of dimension $n$, the minimal class $\frac{\theta^c}{c!}$ is not smoothable for integers $c\leq\frac{n+2}{4}$ satisfying a certain arithmetic condition. For every pair $(d,n)$ with $d\geq\frac{n}{2}$ and such that $c\coloneqq n-d$ satisfies a certain arithmetic condition, Benoist constructs a smooth projective variety of dimension $n$ admitting non-smoothable $d$-cycles in \cite{Olivier2024}. 

Among all the counter-examples mentioned above, the dimension of the cycle is always at least half of the dimension of the ambient variety. When the dimension of the cycle is relatively small compared to the dimension of the variety, there are positive results. Hironaka proves in \cite{Hironaka1968} that~$d$-cycles on a smooth projective variety of dimension $n>2d$ are smoothable when $d\leq 3$. In a recent break through \cite{KollarVoisin2024}, Koll\'ar and Voisin give a complete answer to this question by generalizing Hironaka's result to all~$d\geq 0$.

\begin{thm}[{\cite[Theorem 1.2]{KollarVoisin2024}}]
\label{KV's theorem: smoothability in the Chow ring}
    Let $X$ be a smooth projective variety over $k$ of dimension~$n$. Then, for any $d$ such that $2d<n$, every cycle in $\CH_d(X)$ is smoothable.
\end{thm}

The goal of this article is to prove an analogous result in algebraic cobordism. Our main result is as follows.

\begin{thm}[Theorem \ref{main theorem}]
\label{theorem: main theorem in the introduction}
    Let $X$ be a smooth projective variety over $k$ of dimension $n$. Then classes in the degree $d$ algebraic cobordism group $\Omega_d(X)$ are smoothable when $2d<n$.
\end{thm}

In topology, singular cohomology $H^*$ is refined by complex cobordism $MU^*$, which contains much more information than the former. For example, for each $n\in\bbZ_{\geq 0}$, the group $MU_n(\pt)$ consists of bordism classes of compact stably almost complex manifolds. Two such manifolds $M$ and $N$ represent the same class in $MU_*(\pt)$ if and only if their Chern numbers coincide (see \cite[Chapter VII]{Stong1968}).

In algebraic geometry, we can make an analogy with the topological setting by viewing the Chow ring as the algebraic version of singular cohomology. Then complex cobordism $MU^*$ corresponds to algebraic cobordism $\Omega^*$, which was introduced by Levine and Morel. In \cite{LevineMorel2007}, Levine and Morel define the notion of complex oriented cohomology theories on the category $\Sm_k$ of smooth quasi-projective varieties over $k$, and define $\Omega^*$ to be the universal such theory (see Section \ref{subsection: oriented cohomology theories} and \ref{subsection: Algebraic cobordism and its geometric interpretation}). Since the Chow ring is also a complex oriented cohomology theory, there exists a natural morphism from $\Omega^*$ to~$\CH^*$. Similar to the topological setting, algebraic cobordism $\Omega^*$ contains much more information than the Chow ring~$\CH^*$. In particular, the coefficient ring $\Omega^*(k)$ detects Chern numbers as in the topological setting. Therefore, Theorem \ref{theorem: main theorem in the introduction} refines the original result of Koll\'ar and Voisin, namely Theorem~\ref{KV's theorem: smoothability in the Chow ring}, by incorporating information on Chern numbers of algebraic cycles.

\subsection{The strategy of Koll\'ar and Voisin}

Koll\'ar and Voisin's proof of their Theorem \ref{KV's theorem: smoothability in the Chow ring} is based on new structural results on the Chow ring. More precisely, for $X\in\Sm_k$, they define $\CH^*(X)_{\fl_*\Ch}$ to be the subgroup of~$\CH^*(X)$ generated by cycles of the form $\phi_*z$, where $\phi:Y\ra X$ is a proper and flat morphism from a variety $Y\in\Sm_k$, and $z$ lies in the subring~$\CH^*(Y)_{\Ch}$ of $\CH^*(Y)$ generated by Chern classes of vector bundles on $Y$. 
Koll\'ar and Voisin prove the following surprising result.

\begin{thm}[{\cite[Theorem 1.6]{KollarVoisin2024}}]
\label{KV's theorem: CH(X)=CH(X)_flCh}
    Let $X$ be a smooth projective variety over $k$. Then
    \begin{equation*}
        \CH^*(X)=\CH^*(X)_{\fl_*\Ch}.
    \end{equation*}
\end{thm}

With this result at hand, their main result, Theorem \ref{KV's theorem: smoothability in the Chow ring}, is a direct consequence of the following smoothability result.

\begin{thm}[{\cite[Proposition 1.5]{KollarVoisin2024}}]
\label{KV's theorem: cycles in CH_d(X)_flCh are smoothable when 2d<dim(X)}
    Let $X$ be a smooth variety. Then cycles in $\CH_d(X)_{\fl_*\Ch}$ are smoothable when $2d<\dim(X)$.
\end{thm}

To prove Theorem \ref{KV's theorem: CH(X)=CH(X)_flCh}, they show that $\CH^*(-)_{\fl_*\Ch}$ is stable under pushforwards by blow-ups of smooth \textit{complete bundle-sections} (or ``cbs'' for short). Here, a subvariety $Z\subset X$ of codimension $c$ is a complete bundle-section if $Z$ is the zero-locus of a section of a vector bundle of rank~$c$ over~$X$. The stability result of $\CH^*(-)_{\fl_*\Ch}$ allows them to prove Theorem \ref{KV's theorem: CH(X)=CH(X)_flCh} up to taking (iterated) blow-ups along smooth cbs. They then complete the proof by the following result.

\begin{thm}[{\cite[Corollary 1.10]{KollarVoisin2024}}]
\label{KV's theorem: cbs resolution}
    Let $X$ be a smooth projective variety, and $Z\subset X$ be a smooth subvariety such that $4\dim(Z)<\dim(X)$. Then after successive blow-ups along smooth complete bundle-sections
    \begin{equation*}
        \Pi:X_{r+1}\xra{\tau_{r+1}} X_r\xra{\tau_r}\cdots\xra{\tau_1}X_0\coloneqq X,
    \end{equation*}
    there exists a complete intersection subvariety $Z_{r+1}\subset X_{r+1}$ such that $[Z_{r+1}]$ is mapped to $[Z]$ as cycles under the pushforward $\Pi_*$.
\end{thm}

As for Theorem \ref{KV's theorem: cycles in CH_d(X)_flCh are smoothable when 2d<dim(X)}, its proof essentially reduces to the following result.

\begin{thm}[{\cite[Proposition 2.1]{KollarVoisin2024}}]
\label{KV's theorem: criteria for smoothability}
    Let $\phi:Y\ra X$ be a proper flat morphism in $\Sm_{k}$. Then for any smooth subvariety $Z\subset X$ in general position such that $2\dim(Z)<\dim(X)$, the restriction~$\phi|_Z:Z\ra\phi(Z)$ is an isomorphism. In particular, the subvariety $\phi(Z)\subset X$ is smooth.
\end{thm}

The precise meaning of ``in general position'' is explained in detail in \cite[Section 2]{KollarVoisin2024}. The important cases in Koll\'ar and Voisin's paper \cite{KollarVoisin2024} as well as in our setting are general complete intersections of very ample divisors in a smooth projective variety, which are guaranteed to be in general position by \cite[Remark 2.2]{KollarVoisin2024}.

\subsection{Extending the arguments to algebraic cobordism}

Our proof of Theorem \ref{theorem: main theorem in the introduction} follows a strategy similar to the proof of Koll\'ar and Voisin of their Theorem~\ref{KV's theorem: smoothability in the Chow ring}. In particular, their results on cbs resolution (Theorem \ref{KV's theorem: cbs resolution}) and on criteria for smoothability (Theorem \ref{KV's theorem: criteria for smoothability}) play a central role in our proof. 

However, there are some inevitable difficulties to overcome when working with algebraic cobordism instead of the Chow ring. One of them concerns the formal group law in an oriented cohomology theory which governs the behavior of the first Chern classes of line bundles under the tensor product. In the Chow ring, one has 
$c_1(L\otimes M)=c_1(L)+c_1(M)$ for any line bundles $L$ and $M$ over the same base. Meanwhile, in algebraic cobordism, one has instead 
\begin{equation}
\label{eq: c_1(L*M)=F(c_1(L),c_1(M))}
    c_1(L\otimes M)=F(c_1(L),c_1(M)),
\end{equation}
where $F(u,v)\in\Omega^*(k)[[u,v]]$ is a formal power series. In fact, the graded ring $\Omega^*(k)$ together with the formal power series $F(u,v)$ represents the universal formal group law (see Sections \ref{subsection: Chern classes and Segre classes} and \ref{subsection: Algebraic cobordism and its geometric interpretation}). 

One consequence of this additional complexity is that one has to distinguish carefully between the subgroup $\Omega^*(X)_{\fl_*\rmD}$ of $\Omega^*(X)$ constructed from line bundles and the subgroup $\Omega^*(X)_{\fl_*\rmVAD}$ constructed from very ample line bundles (see Definition \ref{definition of D, AD, f_*D, f_*VAD, s_*D}). We introduce the subgroup $\Omega^*(X)_{\fl_*\rmVAD}$ to construct candidates of smoothable cycles in $\Omega^*(X)$, and we establish the following result. 

\begin{prop}[Proposition \ref{proposition: f_*VAD is smoothable within the Whitney range}]
\label{intro proposition: f_*VAD is smoothable within the Whitney range}
    Let $X$ be a smooth projective variety of dimension $n$ over $k$. Then the cycles in $\Omega_d(X)_{\fl_*\rmVAD}$ are smoothable when $2d<n$.
\end{prop}

The proof of this proposition highly relies on the ampleness condition that appears in the definition of $\Omega^*(X)_{\fl_*\rmVAD}$. It guarantees that every element in this subgroup can be represented as the pushforward of general complete intersections, which are in general positions in the sense of \cite[Remark 2.2]{KollarVoisin2024}. Then we finish the proof by applying Theorem \ref{KV's theorem: criteria for smoothability}.

The subgroup $\Omega^*(X)_{\fl_*\rmD}$ is defined in order to obtain a structural result on $\Omega^*(X)$. We prove the following theorem, which is the analog of Theorem \ref{KV's theorem: CH(X)=CH(X)_flCh} of Koll\'ar and Voisin in algebraic cobordism.

\begin{thm}[Theorem \ref{theorem: all classes are f_*D}]
\label{intro theorem: all classes are f_*D}
    Let $X$ be a smooth projective variety over $k$. Then 
    \begin{equation*}
        \Omega^*(X)=\Omega^*(X)_{\fl_*\rmD}. 
    \end{equation*}
\end{thm}

The first step to prove this theorem is the following result.

\begin{prop}[Proposition \ref{proposition: f_*D is preserved under pushforward along cbs blow-up}]
\label{intro proposition: f_*D is preserved under pushforward along cbs blow-up}
    Let $X$ be a smooth projective variety over $k$, and let $Z\subset X$ be a smooth complete bundle-section in $X$. Let $\tau:\wt{X}\ra X$ be the blow-up of $X$ along $Z$. Then 
    \begin{equation*}
        \tau_*(\Omega^*(\wt{X})_{\fl_*\rmD})\subseteq\Omega^*(X)_{\fl_*\rmD}.
    \end{equation*}
\end{prop}

Proposition \ref{intro proposition: f_*D is preserved under pushforward along cbs blow-up} is the analog of \cite[Proposition 3.11]{KollarVoisin2024} in algebraic cobordism, and the proof follows a similar strategy. We first prove in Proposition \ref{proposition: f_*D is preserved under pushforward of inclusion of smooth hypersurfaces} that $\Omega^*(-)_{\fl_*\rmD}$ is preserved by pushforwards along embeddings of smooth hypersurfaces, which corresponds to \cite[Proposition~3.7]{KollarVoisin2024}. We then use an induction argument to generalize this result to embeddings of smooth cbs and obtain Proposition \ref{proposition: f_*D is preserved under pushforward of inclusion of smooth complete bundle-sections}, which corresponds to \cite[Proposition 3.9]{KollarVoisin2024}. Finally, Proposition \ref{intro proposition: f_*D is preserved under pushforward along cbs blow-up} becomes a direct corollary of Proposition \ref{proposition: f_*D is preserved under pushforward of inclusion of smooth complete bundle-sections} via diagram (\ref{diagram: cbs blow-up}). 

Although the strategies we use to prove Proposition \ref{proposition: f_*D is preserved under pushforward of inclusion of smooth hypersurfaces} and Proposition \ref{proposition: f_*D is preserved under pushforward of inclusion of smooth complete bundle-sections} are similar to those in the Chow ring (see \cite[Section~3.3]{KollarVoisin2024}), additional work is required to extend the argument to algebraic cobordism. For example, in the proof of Proposition \ref{proposition: f_*D is preserved under pushforward of inclusion of smooth hypersurfaces}, the argument in the Chow ring relies on the identity $s_0(E)=1_X$, which no longer holds in algebraic cobordism. Instead, we replace it with equation (\ref{eq: 1_X=c_r(E)s_-r(E)+sum over s with t}), which is established in Proposition \ref{proposition: pushforward from projective bundle}. The same equation is also used in the proof of Proposition \ref{proposition: f_*D is preserved under pushforward of inclusion of smooth complete bundle-sections} to adapt the argument from the Chow ring to algebraic cobordism. For the same purpose, we prove that all Chern classes lie in $\Omega^*(-)_{\sm_*\rmD}$, as stated in Proposition~\ref{proposition: Chern classes belong to s_*D}. In the Chow ring, the analog of this result follows directly from the relation~$c(E)s(E)=1_X$. Since this relation no longer holds in algebraic cobordism, we instead establish Proposition \ref{proposition: L_*-subalgebras generated by Chern classes and Segre classes are the same}, which allows us to deduce Proposition \ref{proposition: Chern classes belong to s_*D} and eventually Proposition \ref{intro proposition: f_*D is preserved under pushforward along cbs blow-up}.

The proof of Theorem \ref{intro theorem: all classes are f_*D} then follows from Proposition \ref{intro proposition: f_*D is preserved under pushforward along cbs blow-up} and the ``cbs resolution theorem'' of Koll\'ar and Voisin (see \cite[Theorem 1.9]{KollarVoisin2024}), by an induction argument on the degree $d$ of the cobordism group~$\Omega_d(X)$ and a comparison result between the Chow ring and algebraic cobordism (see Lemma~\ref{lemma: kernel of theta(X)}).

With Proposition \ref{intro proposition: f_*VAD is smoothable within the Whitney range} and Theorem \ref{intro theorem: all classes are f_*D} at hand, the proof of Theorem \ref{theorem: main theorem in the introduction} reduces to showing that the subgroups $\Omega^*(X)_{\fl_*\rmD}$ and $\Omega^*(X)_{\fl_*\rmVAD}$ of $\Omega^*(X)$ coincide.

\begin{thm}[Theorem \ref{theorem: all classes are fl_*VAD}]
\label{intro theorem: fl_*D=f_*VAD}
    Let $X$ be a smooth projective variety over $k$. Then 
        \begin{equation*}
        \label{intro eq: fl_*D=fl_*VAD}
            \Omega^*(X)_{\fl_*\rmD}=\Omega^*(X)_{\fl_*\rmVAD}.
        \end{equation*}
\end{thm}

The corresponding result in the Chow ring is immediate since for any line bundle $L$, there always exist very ample line bundles $L'$ and $L''$ such that $L\cong L'\otimes L''^\vee$, and hence $c_1(L)=c_1(L')-c_1(L'')$ in the Chow ring. However, in algebraic cobordism, from (\ref{eq: c_1(L*M)=F(c_1(L),c_1(M))}), we can only express $c_1(L)$ as a polynomial in $c_1(L')$ and $c_1(L'')$ with $\bbL^*$ coefficients.
 
Therefore, to prove Theorem \ref{intro theorem: fl_*D=f_*VAD}, it suffices to prove that $\Omega^*(X)_{\fl_*\rmVAD}$ is an $\bbL^*$-submodule of~$\Omega^*(X)$, which is the content of Proposition \ref{proposition: f_*VAD is a L_*-module}. To prove this result, we proceed by induction on both the degree of the coefficients in $\bbL_*$ and the degree of the cycles in $\Omega_*(X)_{\fl_*\rmVAD}$.

As we explained above, one of the main difficulties in generalizing the result of Koll\'ar and Voisin from the Chow ring to algebraic cobordism arises from the complexity of the coefficient ring $\Omega^*(k)$. Unlike~$\CH^*(k)\cong\bbZ$, which is concentrated in degree $0$, the ring $\Omega^*(k)$ is isomorphic to the Lazard ring $\bbL^*$, which classifies all formal group laws. In particular, $\Omega^*(k)$ is isomorphic to a polynomial ring with integer coefficients in countably many variables (see Section \ref{subsection: Chern classes and Segre classes}). The projection to the degree $0$ part of $\bbL^*$ defines a morphism $\bbL^*\ra\bbL^0\cong\bbZ$, which induces an isomorphism 
\begin{equation*}
    \CH^*(X)\cong\Omega^*(X)\otimes_{\bbL^*}\bbZ
\end{equation*}
for any $X\in\Sm_k$ (see Section \ref{subsection: Algebraic cobordism and Chow groups}). Therefore, if two elements $\alpha,\beta\in\Omega_d(X)$ correspond to the same class in $\CH_d(X)$, their difference $\alpha-\beta$ can be highly nontrivial in $\Omega_d(X)$. Nevertheless, it can be written as a linear combination of elements of the form $\lambda\times\gamma$, where  $\lambda\in\bbL_{>0}$ and~$\gamma\in\Omega_{<d}(X)$ (see Lemma~\ref{lemma: kernel of theta(X)}). Since the degree of $\gamma$ is strictly less than that of $\alpha$ and $\beta$, this allows us to apply an induction argument on the degree of classes in $\Omega_*(X)$. This strategy is used in the proofs of Lemma \ref{lemma: generating Chow ring implies generating Omega} and Theorem~\ref{theorem: all classes are f_*D}.

\subsection{Organization of the article}

Section \ref{section: Generalities on algebraic cobordism} serves as a brief introduction to algebraic cobordism. We first introduce the notion of oriented cohomology theories on $\Sm_k$, then define algebraic cobordism to be the universal one, and denote it by $\Omega^*$. We also introduce the Chern classes and Segre classes in $\Omega^*$ analogous to their counterparts in~$\CH^*$. We present a geometric method for constructing cycles in $\Omega^*$, which allows us to compute certain pullbacks and pushforwards explicitly. We end this section with a comparison result between~$\Omega^*(X)$ and $\CH^*(X)$ (see Lemma \ref{lemma: kernel of theta(X)}).

We begin Section \ref{section: Pushforwards of products of line bundles} by defining several subgroups and subrings of $\Omega^*(X)$. We then study their properties, especially their stability under certain pullbacks and pushforwards. In particular, we prove in Proposition \ref{proposition: f_*D is preserved under pushforward along cbs blow-up} that $\Omega^*(X)_{\fl_*\rmD}$ is stable under pushforwards by blow-ups along smooth complete bundle-sections, which is the analog of Koll\'ar and Voisin's result \cite[Proposition~3.11]{KollarVoisin2024} in our context. The proof of this result relies on a relation between Chern classes and Segre classes in algebraic cobordism, which is established in Proposition \ref{proposition: pushforward from projective bundle}. 

In Section \ref{section: Koll\'ar--Voisin cbs resolution and smoothability of cycles in Omega^{*}}, we first prove a structural theorem for $\Omega^*(X)$ in Theorem \ref{theorem: all classes are f_*D}, which asserts that~$\Omega^*(X)=\Omega^*(X)_{\fl_*\rmD}$. We then prove in Proposition \ref{proposition: f_*VAD is smoothable within the Whitney range} that classes in $\Omega_d(X)_{\fl_*\rmVAD}$ are smoothable when~$2d<\dim(X)$. Finally, we deduce our main result, Theorem \ref{main theorem}, by showing that~$\Omega^*(X)_{\fl_*\rmD}=\Omega^*(X)_{\fl_*\rmVAD}$, which is established in Theorem \ref{theorem: all classes are fl_*VAD}. This equality follows from the fact that the subgroup $\Omega^*(X)_{\fl_*\rmVAD}$ of $\Omega^*(X)$ is also a $\bbL^*$-submodule (see Proposition~\ref{proposition: f_*VAD is a L_*-module}).

I would like to thank my advisor, Olivier Benoist, for many enlightening discussions and for his careful reading of preliminary versions of this text.

\section{Generalities on algebraic cobordism}
\label{section: Generalities on algebraic cobordism}

\subsection{Oriented cohomology theories}
\label{subsection: oriented cohomology theories}

Here, we follow the notation of Levine and Morel in \cite{LevineMorel2007}, and denote by $\Sm_k$ the category of smooth quasi-projective complex varieties over $k$, and by~$\R^*$ the category of commutative $\bbZ$-graded rings with a unit.

\begin{Def}
\label{definition of an oriented cohomology theory over Sm_C}
    An oriented cohomology theory consists of the following data:
    \begin{enumerate}[font=\bfseries]
        \item[(D1).] An additive functor $A^*:\Sm_k^{\op}\ra \R^*$.

        \item[(D2).] For each projective morphism $f:Y\ra X$ in $\Sm_k$ of relative dimension $d$, a homomorphism of graded $A^*(X)$-modules:
        \begin{equation*}
            f_*:A^*(Y)\ra A^{*-d}(X).
        \end{equation*}
    \end{enumerate}
    These data are required to satisfy the following properties:
    \begin{enumerate}[align=left, font=\bfseries, leftmargin=*]
        \item[(A1).] For any $X\in\Sm_k$, we have 
        \begin{equation*}
            (\Id_X)_*=\Id_{A^*(X)}.
        \end{equation*}
        Let $f:Y\ra X$ and $g:Z\ra Y$ be projective morphisms in $\Sm_k$ of relative dimensions $d$ and~$e$, respectively. Then we have
        \begin{equation*}
            (f\circ g)_* = f_* \circ g_*\colon A^*(Z)\ra A^{*-d-e}(X).
        \end{equation*}

        \item[(A2).(Transverse commutativity).] Let $f:X\ra Z$ and $g:Y \ra Z$ be two transverse morphisms in $\Sm_k$ in the sense of \cite[Definition 1.1.1.]{LevineMorel2007}. Suppose that $f$ is projective of relative dimension $d$. Then, given a cartesian square 
        \begin{equation*}
            \begin{tikzcd}
                W \ar[r,"g'"] \ar[dr,phantom,"\square"] \ar[d,"f'"'] & X \ar[d,"f"]\\
                Y \ar[r,"g"'] & Z
            \end{tikzcd}
        \end{equation*}
        we have $g^*\circ f_*= f'_*\circ g'^*:A^*(X)\ra A^{*-d}(Y)$.

        \item[(A3).(Projective bundle formula).] Let $X\in\Sm_k$ and $E\ra X$ be a vector bundle over $X$ of rank $r$. Let $\bbP(E)\ra X$ be the projectivization of $E$, and $O(1)\ra\bbP(E)$ be the tautological quotient line bundle. Denote by $s:\bbP(E)\ra O(1)$ the zero section. Define $\zeta\in A^1(\bbP(E))$ by~$\zeta\coloneqq s^*s_*(1)$. Then $A^*(\bbP(E))$ is a free $A^*(X)$-module, with basis $\{1,\zeta,\cdots,\zeta^{r-1}\}$.
        
        \item[(A4).(Extended homotopy property).] Let $X\in\Sm_k$ and $E\ra X$ be a vector bundle over $X$. Let $p:V\ra X$ be an $E$-torsor. Then $p^*:A^*(X)\ra A^*(V)$ is an isomorphism. 
    \end{enumerate}
\end{Def}

\begin{rem}
    Here and throughout, we use Grothendieck's convention: for a vector bundle $E\ra X$, the projectivization $p:\bbP(E)\ra X$ is the projective bundle of one-dimensional quotients of $E$, and~$O(1)\ra\bbP(E)$ denotes the tautological quotient line bundle.
\end{rem}

The second piece of data \tbf{(D2)} implies that the projection formula holds for any projective morphism~$f:Y\ra X$ in $\Sm_k$, $i.e.$, we have
\begin{equation*}
    f_*(f^*\alpha\cdot\beta)=\alpha\cdot f_*\beta,
\end{equation*}
for any $\alpha\in A^*(X)$ and $\beta\in A^*(Y)$.

For any $X\in\Sm_k$, we denote by $1_X\in A^0(X)$ the multiplicative unit of the ring $A^*(X)$. 

We will use the notation $A_*(X)$ for $A^{\dim(X)-*}(X)$. Meanwhile, for a graded ring $R^*$, we will write $R_*$ for $R^{-*}$.

The external product in an oriented cohomology theory $A^*$ is defined as
\begin{align*}
    \times:A^*(X)\times A^*(Y)&\ra A^*(X\times Y)\\
    (\alpha,\beta)&\mapsto \alpha\times\beta\coloneqq pr_X^*\alpha\cdot pr_Y^*\beta
\end{align*}
for any $X,Y\in\Sm_k$, where $pr_X$ and $pr_Y$ are the projections from $X\times Y$ to $X$ and $Y$. In particular, for $\beta=1_Y$, we have $\alpha\times 1_Y=pr_X^*\alpha$. It is straightforward to check that the external product defined in this way is compatible with pushforwards and pullbacks. More precisely, we have 
\begin{equation*}
    (f\times g)_*\circ\times =\times\circ(f_*\times g_*)
\end{equation*}
for any projective morphisms $f,g$ in $\Sm_k$, and 
\begin{equation*}
    (f\times g)^*\circ\times =\times\circ(f^*\times g^*)
\end{equation*}
for any morphisms $f,g$ in $\Sm_k$.

As in the Chow ring, one can also consider correspondences in algebraic cobordism, and we have the following result.

\begin{prop}
\label{proposition: formula of correspondence}
    Let $\phi:Y\ra X$ be a morphism between smooth projective varieties. Denote by~$i:\Gamma_\phi \hra Y\times X$ the embedding of the graph of $\phi$. Then for any class $\alpha\in\Omega^*(Y)$, we have 
    \begin{equation*}
        \phi_*\alpha = pr_{X*}(pr_Y^*\alpha\cdot i_*(1_{\Gamma_\phi})) \text{\ in\ } \Omega^*(X),
    \end{equation*}
    where $pr_Y$ and $pr_X$ are the projections from $Y\times X$ to $Y$ and $X$ respectively.
\end{prop}

\begin{proof}
    Define $p\coloneqq pr_Y\circ i:\Gamma_\phi\ra Y$ and $q\coloneqq pr_X\circ i:\Gamma_\phi\ra X$. Then $p$ is an isomorphism and~$q=\phi\circ p$. By the projection formula, we have 
    \begin{equation*}
        pr_{X*}(pr_Y^*\alpha\cdot i_*(1_{\Gamma_\phi}))= (pr_X\circ i)_*((pr_Y\circ i)^*\alpha)=q_*p^*\alpha=\phi_*p_*p^*\alpha=\phi_*\alpha. \qedhere
    \end{equation*}
\end{proof}

\begin{Def}
\label{definition: morphisms of oriented cohomology theories}
    A morphism of oriented cohomology theories is a natural transformation of functors $\Sm_k^{\op} \ra \R^*$ which commutes with the projective pushforwards.
\end{Def}

\subsection{Chern classes and Segre classes}
\label{subsection: Chern classes and Segre classes}

Given an oriented cohomology theory $A^*$, one can use Grothendieck's method in \cite{Grothendieck1958} to define Chern classes $c_i^A(E)\in A^i(X)$ of a vector bundle $E\ra X$ over $X\in \Sm_{k}$. Suppose that $E$ is of rank $r$. Then, by the projective bundle formula, there exist unique classes $c^A_i(E)\in A^i(X),0\leq i\leq r$, such that $c^A_0(E)=1$ and that
\begin{equation*}
    \sum_{i=0}^r (-1)^i c^A_i(E) \cdot \zeta^{r-i}=0 \text{\ in\ } A^r(\bbP(E)),
\end{equation*}
where $\bbP(E)\ra X$ is the projectivization of $E$ and $\zeta$ is defined as in the projective bundle formula (see Definition \ref{definition of an oriented cohomology theory over Sm_C}). We define $c^A(E) \coloneqq \sum_{i=0}^r c^A_i(E)\in A^*(X)$ to be the total Chern class of $E$. When the cohomology theory $A^*$ is clear from the context, we drop the superscript and simply write~$c_i(E)$ instead.

One can check that the Chern classes defined above satisfy the standard properties. We collect some of them in the following proposition. For a detailed discussion, we refer to the book by Levine and Morel \cite[Section 1.1, Section 4.1.7]{LevineMorel2007}.

\begin{prop}
    Let $A^*$ be an oriented cohomology theory, and $c_i(-)$ be the Chern classes in $A^*$. Then they satisfy the following properties.
    \begin{enumerate}[itemsep=0em]
        \item Let $L\ra X$ be a line bundle over $X\in\Sm_k$, and let $s:X\hra L$ be the zero section. Then 
        \begin{equation*}
            c_1(L)=s^*s_*(1_X).
        \end{equation*}

        \item For any morphism $f:Y \ra X$ in $\Sm_k$, and any vector bundle $E\ra X$ over $X$, we have 
        \begin{equation*}
            c(f^*E)=f^*c(E).
        \end{equation*}

        \item If $0 \ra E' \ra E \ra E''\ra 0$ is a short exact sequence of vector bundles over $X\in\Sm_k$, we have 
        \begin{equation*}
            c(E)=c(E')\cdot c(E'').
        \end{equation*}
    \end{enumerate}
\end{prop}

One important example of an oriented cohomology theory is the Chow ring. The proposition above shows that many of the properties satisfied by the Chern classes in the Chow ring still hold for any oriented cohomology theory. However, there is a significant difference when tensor products of vector bundles come into play. 

Let $L$ and $M$ be two line bundles over the same base $X\in\Sm_k$. Then, in the Chow ring, one has $c_1(L\otimes M)=c_1(L) + c_1(M)$. This identity no longer holds in general for arbitrary oriented cohomology theories, as we now explain. 

\begin{Def}
    A commutative formal group law of rank one is a pair $(R,F)$ consisting of a commutative ring $R$ and a formal power series
    \begin{equation*}
        F(u,v) = \sum_{i,j\geq 0}r_{i,j}u^iv^j\in R[[u,v]]
    \end{equation*}
    satisfying the following identities: 
    \begin{enumerate}[itemsep=0em]
        \item $F(u,0) = F(0,u) = u \in R[[u]]$.
        \item $F(u,v) = F(v,u) \in R[[u,v]]$.
        \item $F(F(u,v),w) = F(u,F(v,w)) \in R[[u,v,w]]$.  
    \end{enumerate}
    When $R=R^*$ is a graded ring, we say that $F$ is graded if its coefficients satisfy $r_{i,j}\in R^{1-i-j}$ for all $i,j\geq 0$, so that we have $F(u,v)\in R^1$ whenever $u,v\in R^1$. We will use the term \textit{graded formal group law} in what follows to refer to such an $F$.
\end{Def}

\begin{rem}
    The first two conditions imply that $F(u,v)=u+v+\sum_{i,j\geq 1}r_{i,j}u^i v^j$ and $r_{i,j}=r_{j,i}$ for all $i,j\geq 1$. The last condition gives rise to many relations among the $r_{i,j}$ which are too complicated to write down explicitly.
\end{rem}

It is not difficult to prove that there exists a universal graded formal group law $(\bbL^*,F)$. This means that, for any graded formal group law $F_{R^*}$ over a graded ring $R^*$, there exists a unique morphism $\phi:\bbL^*\ra R^*$ of graded commutative rings such that $F$ is sent to $F_{R^*}$ under $\phi$.

Lazard proved in \cite{Lazard1955} that the ring $\bbL^*$ is a polynomial ring over $\bbZ$ in countably many variables, namely, 
\begin{equation*}
    \bbL^*\cong\bbZ[x_1,x_2,\ldots],
\end{equation*}
where the degree of $x_i$ is $-i$. The ring $\bbL^*$ is now called the Lazard ring.

As pointed out by Levine and Morel in their book \cite[Lemma 1.1.3]{LevineMorel2007}, the first Chern class of the tensor product of two line bundles defines a graded formal group law over the coefficient ring~$A^*(k)$. To be more concrete, we have the following result.

Let $A^*$ be an oriented cohomology theory, and let $L\ra X$ be a line bundle over $X\in \Sm_k$. Then~$c_1(L)^m$ vanishes for all $m>\dim(X)$. Moreover, there exists a unique formal power series
\begin{equation*}
    F_A(u,v) = \sum_{i,j\geq 0} a_{i,j}u^iv^j \in A^*(k)[[u,v]]
\end{equation*}
such that, for any $X\in\Sm_k$ and any line bundles $L,M$ over $X$, we have 
\begin{equation*}
    c_1(L \otimes M) = F_A(c_1(L),c_1(M)) \text{\ in\ } A^1(X).
\end{equation*}
From the properties of the first Chern class, it follows that $F_A$ defines a graded formal group law over $A^*(k)$. By the universality of the Lazard ring, there exists a unique morphism $\bbL^*\ra A^*(k)$ of graded rings that maps $F$ to $F_A$. We will write $\alpha+_A\beta$ for $F_A(\alpha,\beta)$, where $\alpha,\beta\in A^*(X)$. In particular, we have $c_1(L\otimes M)=c_1(L)+_A c_1(M)$.

As pointed out by Levine and Morel in \cite[Remark 4.1.2]{LevineMorel2007}, the splitting principle also holds in any oriented cohomology theory $A^*$. More precisely, for any vector bundle $E\ra X$ of rank $r$ over~$X\in\Sm_k$, there is a smooth morphism $f:X'\ra X$ such that
\begin{enumerate}[itemsep=0em]
    \item The bundle $f^*E$ splits as a direct sum of $r$ line bundles $L_i$, $1\leq i\leq r$, over $X'$.
    \item The morphism $f^*:A^*(X)\ra A^*(X')$ is injective.
\end{enumerate}
Therefore, if we denote $c_1(L_i)\in A^1(X')$ by $\alpha_i$, we have 
\begin{equation*}
    f^*c(E)=c(f^*E)=\prod_{i=1}^r c(L_i)=\prod_{i=1}^r(1+\alpha_i) \text{\ in\ } A^*(X').
\end{equation*}
As usual, we call $\alpha_i$ the Chern roots of $E$. Then $c_i(E)$ is identified with the $i$-th elementary symmetric polynomial of these Chern roots under the morphism $f^*$.

Apart from Chern classes, we can also define Segre classes in any oriented cohomology theory. 

\begin{Def}
    Let $A^*$ be an oriented cohomology theory. Let $X\in\Sm_k$, and let $E\ra X$ be a vector bundle of rank $r$. For any $i\in\bbZ$, let $l\in\bbZ$ satisfies $l\geq \max(0,1-i-r)$, and let~$p:\bbP((E\oplus O_X^{\oplus l})^\vee)\ra X$ be the projectivization of the dual bundle of $E\oplus O_X^{\oplus l}\ra X$. The $i$-th Segre class~$s^A_i(E)\in A^i(X)$ of $E$ is defined by 
    \begin{equation*}
        s^A_i(E)\coloneqq p_*(c_1(O(1))^{i+r+l-1}),
    \end{equation*}
    where $O(1)\ra \bbP((E\oplus O_X^{\oplus l})^\vee)$ is the tautological quotient line bundle.

    We define $s^A(E)\coloneqq\sum_{i\in\bbZ}s^A_i(E)$ to be the total Segre class. One should be aware that it is an element of $\prod_{i\in\bbZ}A^i(X)$ instead of $A^*(X)$. 
\end{Def}

\begin{rem}
    One checks that the definition above does not depend on the choice of $l$. We will omit the superscript and use only $s_i(E)$ when the cohomology theory $A^*$ is clear from the context.
\end{rem}

\begin{rem}
    In contrast with the case of the Chow ring, Segre classes with negative indices may be non-trivial in a general oriented cohomology theory $A^*$. Moreover, the class $s^A_0(E)$ is not necessarily equal to $1_X$, whereas in the Chow ring one has $s_0(E)=1_X$.
\end{rem}

\begin{prop}
\label{proposition: morphism of oriented cohomology theories preserves Chern classes and Segre classes}
    Let $\theta:A^*\ra B^*$ be a morphism of oriented cohomology theories, then for any vector bundle $E\ra X$ over $X\in \Sm_k$, we have $\theta(c^A(E))=c^B(E)$, and $\theta(s^A(E))=s^B(E)$.
\end{prop}

\begin{proof}
    Since $\theta$ commutes with pushforwards and pullbacks, it preserves Segre classes by definition. In what follows, we will prove that it preserves Chern classes as well.

    Suppose that the rank of $E$ is $r$. Let $p:\bbP(E)\ra X$ be the projective bundle associated to $E$, and let $O(1)\ra\bbP(E)$ be the tautological quotient line bundle. Let $s:X\ra E$ be the zero section. We define $\zeta^A\in A^1(\bbP(E))$ by $\zeta^A=s^*s_*(1)$, and define $\zeta^B$ similarly. Since $\theta$ commutes with pullbacks and pushforwards, we have $\theta(\zeta^A)=\zeta^B$. 

    By the definition of Chern classes in $A^*$, we have $c^A_0(E)=1^A_X$ and 
    \begin{equation}
    \label{eq: def of Chern classes in A}
        \sum_{i\geq 0}(-1)^i c^A_i(E)\cdot(\zeta^A)^{r-i}=0 \text{\ in\ } A^r(\bbP(E)).
    \end{equation}

    After applying $\theta$ to equation (\ref{eq: def of Chern classes in A}) and using $\theta(\zeta^A)=\zeta^B$, we obtain
    \begin{equation}
    \label{eq: def of Chern classes in B}
        \sum_{i\geq 0}(-1)^i\theta(c^A_i(E))\cdot(\zeta^B)^{r-i}=0 \text{\ in\ } B^r(\bbP(E)).
    \end{equation}
    Since $\theta(c^A_0(E))=\theta(1^A_X)=1^B_X=c^B_0(E)$, comparing equation (\ref{eq: def of Chern classes in B}) with the definition of~$c^B_i(E)$ and using the uniqueness of the Chern classes, we have $\theta(c^A_i(E))=c^B_i(E)$ for all $0\leq i\leq r$.
\end{proof}

\subsection{Algebraic cobordism and its geometric interpretation}
\label{subsection: Algebraic cobordism and its geometric interpretation}

As shown in \cite{LevineMorel2007}, there exists a unique universal oriented cohomology theory, called algebraic cobordism and denoted by $\Omega^*$. Recall that for any oriented cohomology theory $A^*$, there exists a natural morphism $\bbL^*\ra A^*(k)$ which maps the universal formal group law $F$ to $F_A$. In the case that $A^*=\Omega^*$, the natural morphism $\bbL^*\ra\Omega^*(k)$ turns out to be an isomorphism by \cite[Theorem 1.2.7]{LevineMorel2007}. We will not distinguish these two rings in the following.

To any projective morphism $\phi:Y\ra X$ in $\Sm_k$, one can associate a class $\phi_*(1_Y)$ in $\Omega^*(X)$, which we denote by 
\begin{equation*}
    [\phi:Y\ra X]\coloneqq \phi_*(1_Y)\in\Omega_{\dim(Y)}(X)=\Omega^{\dim(X)-\dim(Y)}(X).
\end{equation*}
Classes of this form are called standard cycles.

Levine and Morel show that for any $X\in \Sm_k$, the group $\Omega^*(X)$ is generated by standard cycles (see \cite[Lemma 2.5.11]{LevineMorel2007}).
In particular, we obtain $\Omega^i(X)=\Omega_{\dim(X)-i}=0$ when~$i>\dim(X)$. When $X=\Spec(k)$, the group $\Omega_d(k)$ is generated by classes $[Z/k]\coloneqq [\pi_Z:Z\ra \Spec(k)]$, where~$Z$ is a smooth projective variety of dimension $d$ over $k$, and $\pi_Z$ is the structure morphism of $Z$.

Using these generators, projective pushforwards and smooth pullbacks can be described explicitly. Let $f:X\ra X'$ be a projective morphism in $\Sm_k$. Then the pushforward $f_*:\Omega_*(X)\ra\Omega_*(X')$ sends the standard cycle $[\phi:Y\ra X]$ to the standard cycle $[f\circ\phi:Y\ra X']$. Similarly, the pullback $f^*:\Omega^*(X)\ra\Omega^*(X')$ of a smooth morphism $f:X'\ra X$ in $\Sm_k$ sends the standard cycle~$[\phi:Y\ra X]$ to the standard cycle~$[pr_{X'}:Y\times_X X'\ra X']$. The pullbacks of general morphisms in $\Sm_k$ require a more involved construction, which can be found in \cite[Section~6.5]{LevineMorel2007}.

The external product can also be described explicitly. Let $[f:Y\ra X]$ and $[f':Y'\ra X']$ be two standard cycles in $\Omega^*(X)$ and $\Omega^*(X')$ respectively. Their external product in $\Omega^*(X\times X')$ is given by the standard cycle $[f\times f':Y\times Y'\ra X\times X']$. 

Some Chern classes can also be rewritten in a similar manner once the vector bundle admits a transverse section. More precisely, we have the following result.

\begin{prop}
\label{proposition: geometric interpretation of top Chern class}
    Let $X\in\Sm_k$, and $E\ra X$ be a vector bundle of rank $r$ over $X$. Suppose that $\sigma\in\Gamma(X,E)$ is a transverse section of $E$, with zero-locus $Y\subset X$. Then $Y$ is a smooth closed subvariety of codimension $r$ in $X$, and we have $[Y\hra X]=c_r(E)$ in $\Omega^r(X)$.
\end{prop}

Although it seems to be a well-known result, its proof does not seem to appear in the existing literature except for $r=1$.
Before presenting a complete proof of the general case, we consider the following construction that will be used repeatedly in the following.

Suppose that $Y$ is the zero-locus of a transverse section $\sigma$ of a vector bundle $\pi:E\ra X$ of rank~$r$. Denote by $p: \bbP(E) \ra X$ the projectivization of $E$. Let $O(1)\ra\bbP(E)$ be the tautological quotient line bundle, and let $F$ be the kernel of the quotient bundle morphism $p^*E \ra O(1)$. Then the vector bundle $F\ra\bbP(E)$ is of rank $r-1$.
    
The section $\sigma$ of $E$ induces a section $p^*\sigma$ of $p^*E$, which projects to a section $\sigma'$ of $O(1)$. We denote by $X' \subset \bbP(E)$ the zero-locus of $\sigma'$. Then over $X'$, the section $\sigma'$ induces a section $\sigma''$ of the restriction $F|_{X'}\ra X'$.

Koll\'ar and Voisin proved that $X'$ is a smooth hypersurface of $\bbP(E)$, and that the section $\sigma''$ is transverse with zero-locus $p^{-1}(Y) = \bbP(E|_Y)$ (see \cite[Lemma 3.10]{KollarVoisin2024}). We obtain the following commutative diagram.

\begin{equation}
\label{diagram: cbs diagram}
    \begin{tikzcd}
        O(1) \ar[dr] & p^*E \ar[d] \ar[l, two heads] & E \ar[d]\\
        F|_{X'} \ar[dr] & \bbP(E) \ar[ddr,phantom,"\square"] \ar[u, bend right, "p^*\sigma"'] \ar[ul, bend left, looseness=0.5, "\sigma'"] \ar[r, "p"] & X \ar[u, bend right, "\sigma"']\\
        & \{\sigma'=0\}=X' \ar[u,hook,"j''"]\ar[ul, bend left, looseness=0.5, "\sigma''"] & \\
        & \{\sigma''=0\}=p^{-1}(Y)=\bbP(E|_Y) \ar[u, hook,"j'"] \ar[r, "p|_Y"] & Y \ar[uu,hook,"j"']
    \end{tikzcd}
\end{equation}

Using this commutative diagram, we can prove Proposition \ref{proposition: geometric interpretation of top Chern class} by induction on the rank $r$ of the vector bundle $E$.

\begin{proof}[Proof of Proposition \ref{proposition: geometric interpretation of top Chern class}]
    Denote by $s_0:X\hra E$ the zero-section of $E$. 
    
    When~$r=1$, we consider the following Cartesian diagram.
    \begin{equation*}
        \begin{tikzcd}
            Y \ar[r,"i"] \ar[d,"i"'] \ar[dr,phantom,"\square"] & X \ar[d,"s_0"]\\
            X  \ar[r,"s"']& E
        \end{tikzcd}
    \end{equation*}
    By transverse commutativity (see Definition \ref{definition of an oriented cohomology theory over Sm_C}(A2)), we have 
    \begin{equation*}
        [i:Y\hra X]=i_*(1_Y)=i_*i^*(1_X)=s^*s_{0*}(1_X)=s_0^*s_{0*}(1_X)=c_1(E).
    \end{equation*}
    Here, we have $s^*=s_0^*$ since both are inverses of the isomorphism $\pi^*:\Omega^1(X)\cong\Omega^1(E)$.

    When $r\geq 2$, we proceed by induction on $r$. We suppose that the statement holds for any vector bundle of rank $r-1$. Apply transverse commutativity to diagram (\ref{diagram: cbs diagram}), we obtain 
    \begin{equation}
    \label{eq: p^*[Y hra X]=j''_*j'_*1}
        p^*[Y\hra X]=p^*j_*(1_Y)=j''_*j'_*(1_{p^{-1}(Y)}).
    \end{equation} 
    Since $p^{-1}(Y)\subset X'$ is the zero-locus of the transverse section $\sigma''$ of the vector bundle $F|_{X'}$, which is of rank $r-1$ over $X'$, we may apply the induction hypothesis and get 
    \begin{equation*}
        j'_*(1_{p^{-1}(Y)})=c_{r-1}(F|_{X'})=j''^*c_{r-1}(F).
    \end{equation*} 
    Therefore, equation (\ref{eq: p^*[Y hra X]=j''_*j'_*1}) becomes
    \begin{equation}
    \label{eq: p^*[Y hra X]=c_r-1(F)j''_*(1_X')}
        p^*[Y\hra X]=j''_*j''^*c_{r-1}(F)=c_{r-1}(F)\cdot j''_*(1_{X'}).
    \end{equation}
    Since $j'':X'\hra \bbP(E)$ is the inclusion of a smooth hypersurface defined as the zero-locus of the transverse section $\sigma'$, we have $j''_*(1_{X'})=c_1(O(1))$ by the case $r=1$. So equation (\ref{eq: p^*[Y hra X]=c_r-1(F)j''_*(1_X')}) becomes 
    \begin{equation*}
        p^*[Y\hra X]=c_{r-1}(F)\cdot c_1(O(1))=c_r(p^*E)=p^*c_r(E).
    \end{equation*}
    Since $p^*:\Omega^*(X)\ra\Omega^*(\bbP(E))$ is injective, we obtain $[Y\hra X]=c_r(E)$. This completes the induction step and hence the proof.
\end{proof}

Let $Y$ be a smooth complete bundle section in $X$. Suppose that $Y$ is defined as the zero-locus of a section $\sigma$ of a vector bundle $E\ra X$ of rank $r$. Since $Y$ is smooth, the section $\sigma$ has to be transverse. Therefore, by the proposition above, we have $[Y\hra X]=c_r(E)$ in $\Omega^r(X)$. In particular, for every smooth hypersurface $Y$ in $X$, we can identify $[Y\hra X]$ with $c_1(O_X(Y))$ in $\Omega^1(X)$. 

\subsection{Algebraic cobordism and Chow groups}
\label{subsection: Algebraic cobordism and Chow groups}

Since the Chow ring is an oriented cohomology theory and algebraic cobordism is universal among such theories, there exists a natural morphism 
\begin{equation*}
    \theta(X):\Omega^*(X) \ra \CH^*(X),
\end{equation*}
for any $X\in\Sm_k$, which commutes with both pullbacks and projective pushforwards. By Proposition \ref{proposition: morphism of oriented cohomology theories preserves Chern classes and Segre classes}, it sends Chern classes and Segre classes in algebraic cobordism to the corresponding classes in the Chow ring. 

When $X=\Spec(k)$, the morphism $\theta(X)$ becomes a morphism $\bbL^*\ra \bbZ$ of graded rings, where~$\bbZ$ is concentrated in degree zero. This morphism is given by the composition of the projection $\bbL^*\ra\bbL^0$ with the isomorphism~$\bbL^0\cong\bbZ$. Consequently, for any $X\in\Sm_k$, the morphism $\theta(X)$ induces a morphism 
\begin{equation}
\label{eq: Omega^*(X)otimes_bbL bbZ cong CH^*(X)}
    \Omega^*(X)\otimes_{\bbL^*}\bbZ\ra \CH^*(X),
\end{equation}
which is an isomorphism by \cite[Theorem 1.2.19]{LevineMorel2007}. This leads to the following lemma.

\begin{lem}
\label{lemma: kernel of theta(X)}
    Let $X\in\Sm_k$, then the natural morphism $\theta(X):\Omega^*(X)\ra \CH^*(X)$ is always surjective, and the kernel is given by the image of the morphism $\Omega^*(X)\otimes_{\bbL_*}\bbL_{>0}\ra \Omega^*(X)$, denoted by $\Omega^*(X)\times\bbL_{>0}$. It is an $\bbL_*$-submodule of $\Omega^*(X)$.
\end{lem}

\begin{proof}
    The morphism $\bbL_*\ra \bbL_0\cong\bbZ$ fits into a short exact sequence of $\bbL_*$-modules
    \begin{equation*}
        0\ra \bbL_{>0}\ra\bbL_*\ra \bbZ \ra 0.
    \end{equation*}
    Tensoring with $\Omega^*(X)$ over $\bbL_*$, we obtain an exact sequence
    \begin{equation*}
        \Omega^*(X)\otimes_{\bbL_*} \bbL_{>0} \ra \Omega^*(X) \ra \Omega^*(X)\otimes_{\bbL_*}\bbZ\cong \CH^*(X) \ra 0,
    \end{equation*}
    where the isomorphism follows from (\ref{eq: Omega^*(X)otimes_bbL bbZ cong CH^*(X)}), and the middle morphism is precisely $\theta(X)$. The lemma follows immediately from this exact sequence.
\end{proof}

\begin{lem}
\label{lemma: generating Chow ring implies generating Omega}
    Let $X\in\Sm_k$, and let $\theta\coloneqq\theta(X):\Omega^*(X) \ra \CH^*(X)$ be the natural morphism. If a set of classes $\{\alpha_{\lambda},\ \lambda \in \Lambda\}$ satisfies that $\{\theta(\alpha_\lambda),\ \lambda\in\Lambda\}$ generates $\CH^*(X)$ as a ring, then the set~$\{\alpha_{\lambda},
    \ \lambda \in \Lambda\}$ generates $\Omega^*(X)$ as an $\bbL^*$-algebra.
\end{lem}

\begin{proof}
    Let $S\subseteq \Omega^*(X)$ be the $\bbL^*$-subalgebra generated by $\{\alpha_\lambda,\ \lambda\in\Lambda\}$. Then it suffices to show that $\Omega_d(X)\subseteq S$ for all $d\in\bbZ$. We prove this by induction on $d$. The case $d<0$ is trivial since~$\Omega_d(X)=0$. Assume that $d\geq 0$ and that $\Omega_{<d}(X)\subseteq S$. We will prove that $\Omega_d(X)\subseteq S$.

    Let $\beta$ be an arbitrary class in $\Omega_d(X)$. Since classes $\{\theta(\alpha_\lambda),\ \lambda\in\Lambda\}$ generate $\CH^*(X)$ as a ring, we have $\theta(\beta)=P(\theta(\alpha_i),\ i\in I)$, where $I$ is a finite subset of $\Lambda$ and $P$ is a polynomial with integer coefficients. Define $\alpha\in \Omega_d(X)$ by $\alpha=P(\alpha_i, i\in I)$. Then $\alpha\in S$ and satisfies $\theta(\alpha)=\theta(\beta)$. Therefore, by Lemma \ref{lemma: kernel of theta(X)}, the class $\alpha-\beta$ lies in $\Omega^*(X)\times\bbL_{>0}$. So we have $\alpha-\beta =\sum_{j\in J}a_j\beta_j$ for some~$a_j\in\bbL_{>0}$ and $\beta_j\in\Omega_{<d}(X)$, indexed by a finite set $J$. By the induction hypothesis, each~$\beta_j$ is in $S$, hence so is $\alpha-\beta$. Since $\alpha$ is in $S$ by construction, we have $\beta\in S$, which completes the induction step.
\end{proof}

Section \ref{subsection: Chern classes and Segre classes} defines Chern classes and Segre classes in any oriented cohomology theory. It is well known that for any vector bundle $E\ra X$, we have $c(E)\cdot s(E)=1$ in $\CH^*(X)$. However, this equality does not hold in general in algebraic cobordism. Nevertheless, there are still some relations between them. One such relation is the following result.

\begin{prop}
\label{proposition: L_*-subalgebras generated by Chern classes and Segre classes are the same}
    Let $X\in\Sm_k$, and let $E\ra X$ be a vector bundle. Then the $\bbL^*$-subalgebra of~$\Omega^*(X)$ generated by the Chern classes $c_i(E)$ coincides with the $\bbL^*$-subalgebra generated by the Segre classes $s_i(E)$.
\end{prop}

\begin{proof}
    By Jouanolou's trick \cite[Lemma 1.5]{Jouanolou1973}, there exists a vector bundle $F\ra X$ together with an $F$-torsor $f:V\ra X$ such that $V$ is affine. Since the pullback $f^*:\Omega^*(X)\ra\Omega^*(V)$ is an isomorphism by the extented homotopy property (see Definition \ref{definition of an oriented cohomology theory over Sm_C}), it suffices, by naturality of Segre and Chern classes, to prove the proposition for the vector bundle $f^*E\ra V$. As $V$ is affine, the bundle~$f^*E$ is globally generated. Therefore, we may assume that the bundle $E\ra X$ is globally generated.

    Let $C(E)$ and $S(E)$ be the $\bbL^*$-subalgebras of $\Omega^*(X)$ generated by the Chern classes and the Segre classes of $E$ respectively. Then it suffices to show that $C(E)=S(E)$.

    Suppose that the rank of the bundle $E$ is $r$. Since $E$ is globally generated, there exists a classifying morphism $\phi:X\ra \Gr(r,N)=B$ for some integer $N\gg 0$, such that $E\cong \phi^*T$, where $T\ra B$ is the tautological bundle over $B$. Since~$c(E)=\phi^*c(T)$ and $s(E)=\phi^*s(T)$, we have $C(X)=\phi^* C(B)$ and $S(X)=\phi^*S(B)$. Thus it suffices to prove the statement for the tautological bundle $T\ra B$. 

    We denote by $Q\ra B$ the universal quotient bundle over $B$. By Proposition \ref{proposition: morphism of oriented cohomology theories preserves Chern classes and Segre classes}, the natural morphism $\theta(B):\Omega^*(B)\ra \CH^*(B)$ maps $c_i(T)$ to $c^{\CH}_i(T)$ and $s_i(T)$ to $s^{\CH}_i(T)$, and similarly for the Chern classes and Segre classes of the bundle $Q\ra B$. Here and in the sequel, we omit the superscript $\Omega$ when referring to Chern classes and Segree classes in algebraic cobordism.

    Since the Chow ring of $B$ is given by 
    \begin{equation*}
        \CH^*(B) \cong \frac{\bbZ[c^{\CH}_i(T), c^{\CH}_j(Q)]}{(c^{\CH}(T)\cdot c^{\CH}(Q)-1)},
    \end{equation*} 
    the two sets of classes $\{\theta(c_i(T))=c^{\CH}_i(T),\ i\in\bbZ\}$ and $\{\theta(s_i(T))=s^{\CH}_i(T)=c^{\CH}_i(Q),\ i\in\bbZ\}$ both generate $\CH^*(B)$ as a ring. So by Lemma \ref{lemma: generating Chow ring implies generating Omega}, both sets $\{c_i(T),\ i\in\bbZ\}$ and $\{s_i(T),\ i\in\bbZ\}$ generate~$\Omega^*(B)$ as an $\bbL^*$-algebra. Therefore, we have $C(B)=\Omega^*(B)=S(B)$, which completes the proof.
\end{proof}

\section{Pushforwards of products of line bundles}
\label{section: Pushforwards of products of line bundles}

\subsection{Some subgroups of algebraic cobordism}
\label{subsection: Some subgroups of algebraic cobordism}

In this section, we begin by defining several subgroups and subrings of algebraic cobordism from first Chern classes of line bundles. Then we prove some basic properties of them, especially their behavior under certain pullbacks and pushforwards.

\begin{Def}
\label{definition of D, AD, f_*D, f_*VAD, s_*D}
    Let $X$ be a smooth projective variety. We define $\Omega^*(X)_\rmD$ to be the subring of~$\Omega^*(X)$ generated by the first Chern classes of line bundles on $X$, $i.e.$, elements in $\Omega^*(X)_\rmD$ are linear combinations of classes of the form $\prod_{i=1}^s c_1(L_i)$ where $L_i\ra X$ are line bundles over $X$. 
    
    Similarly, we define $\Omega^*(X)_{\rmVAD}$ to be the subring of $\Omega^*(X)$ generated by first Chern classes of very ample line bundles on $X$, $i.e.$, we further require $L_i\ra X$ to be very ample.

    We define $\Omega^*(X)_{\fl_*\rmD}$ to be the subgroup of $\Omega^*(X)$ generated by elements of the form $\phi_*\alpha$, where~$\phi:Y\ra X$ is a flat morphism from a smooth projective variety $Y$ to $X$, and $\alpha\in\Omega^*(Y)_\rmD$. We define $\Omega^*(X)_{\rmsm_*\rmD}$ similarly by requiring the morphism $\phi$ to be smooth, and define $\Omega^*(X)_{\fl_*\rmVAD}$ similarly by requiring $\alpha\in\Omega^*(Y)_{\rmVAD}$.
\end{Def}

Let $X$ be a smooth projective variety, and let $D$ be a Cartier divisor on $X$. We denote by~$[D]_\Omega$ the class $c_1(O_X(D))\in\Omega^1(X)_\rmD$, where $O_X(D)\ra X$ is the line bundle associated to $D$. In this way, we obtain a map 
\begin{equation*}
    [\ ]_\Omega:\Pic(X)\ra\Omega^1(X)_\rmD.
\end{equation*}
However, this map is not a group morphism since the map $c_1:\Pic(X)\ra\Omega^1(X)$ is not. 

By the naturality of Chern classes, we have $\phi^*(\Omega^*(X)_\rmD)\subseteq \Omega^*(Y)_\rmD$ for any morphism $\phi:Y\ra X$ in $\Sm_k$. In what follows, we study the properties of $\Omega^*(X)_{\rmsm_*\rmD}$ and $\Omega^*(X)_{\fl_*\rmD}$.

\begin{prop}
\label{proposition: s_*D is a L_*-algebra and f_*D is a s_*D-module}
    Let $X$ be a smooth projective variety over $k$, then 
    \begin{enumerate}[itemsep=0em]
        \item $\Omega^*(X)_{\rmsm_*\rmD}$ is an $\bbL^*$-algebra,
        \item $\Omega^*(X)_{\fl_*\rmD}$ is a $\Omega^*(X)_{\rmsm_*\rmD}$-module. 
    \end{enumerate}
\end{prop}

\begin{proof}
    First, we show that $\Omega^*(X)_{\rmsm_*\rmD}$ is a subring of $\Omega^*(X)$ and that $\Omega^*(X)_{\fl_*\rmD}$ is a module over it simultaneously. 
    Let $\alpha\in\Omega^*(X)_{\rmsm_*\rmD}$, and $\beta\in\Omega^*(X)_{\fl_*\rmD}$ (respectively $\Omega^*(X)_{\rmsm_*\rmD}$). By definition, we have $\alpha=f_*\alpha'$ for some smooth morphism $f:Y\ra X$ from $Y$ smooth and projective, and~$\alpha'\in\Omega^*(Y)_\rmD$. We have $\beta=g_*\beta'$ for some flat (respectively smooth) morphism $g:Z\ra X$ from~$Z$ smooth and projective, and $\beta'\in\Omega^*(Z)_\rmD$. We consider the Cartesian square formed by $f$ and $g$.
    \begin{equation*}
        \begin{tikzcd}
            W \ar[r,"f'"] \ar[d,"g'"'] \ar[dr, phantom, "\square"] & Z \ar[d,"g"]\\
            Y \ar[r,"f"'] & X
        \end{tikzcd}
    \end{equation*}
    Since $f$ is smooth, the morphisms $f$ and $g$ are transverse, and we have 
    \begin{equation}
    \label{eq: alpha cdot beta}
        \alpha\cdot\beta = f_*\alpha'\cdot g_*\beta' = f_*(\alpha'\cdot f^*g_*\beta') = f_*(\alpha'\cdot g'_*f'^*\beta') = f_*g'_*(g'^*\alpha'\cdot f'^*\beta').
    \end{equation}
    Here, the second and the last equalities follow from the projection formula, while the third follows from the transverse commutativity.

    Since $\alpha'\in\Omega^*(Y)_\rmD$, we have $g'^*\alpha'\in\Omega^*(W)_\rmD$, and similarly for $f'^*\beta'$ and their products. Meanwhile, the morphism $f\circ g'$ is flat (respectively smooth) since $f$ is smooth and $g'$ is flat (respectively smooth). Therefore, the element $\alpha\cdot\beta$ is in $\Omega^*(X)_{\fl_*\rmD}$ (respectively $\Omega^*(X)_{\rmsm_*\rmD}$) by (\ref{eq: alpha cdot beta}).

    So we have proved that $\Omega^*(X)_{\rmsm_*\rmD}$ is a subring of $\Omega^*(X)$, and $\Omega^*(X)_{\fl_*\rmD}$ is a module over it. To conclude, we prove that the image of $\pi_X^*:\bbL^*\ra\Omega^*(X)$ actually lies in $\Omega^*(X)_{\rmsm_*\rmD}$. Let~$a\in\bbL^*$. We may assume that $a=[Z/k]$ is represented by a smooth projective variety~$Z$. Then $\pi_X^*(a)$ is represented by the projection $[pr_X:Z\times X\ra X]$, hence equal to $pr_{X*}(1_{Z\times X})$. Since $pr_X$ is smooth and $1_{Z\times X}$ is in $\Omega^*(Z\times X)_\rmD$, the class $\pi_X^*(a)$ is in $\Omega^*(X)_{\rmsm_*\rmD}$ by definition.
\end{proof}

\begin{rem}
\label{remark: f_*D is a L_*-module and an Omega^*(X)_D-module}
    Let $X$ be a smooth projective variety. Then $\Omega^*(X)_{\fl_*\rmD}$ is naturally an $\bbL^*$-module and an $\Omega^*(X)_\rmD$-module, since both $\bbL^*$ and $\Omega^*(X)_\rmD$ admit natural morphisms to $\Omega^*(X)_{\rmsm_*\rmD}$.
\end{rem}

The group $\Omega^*(X)_{\fl_*\rmD}$ satisfies the following functorial properties.

\begin{prop}
\label{proposition: stability results for f_*D}
    Let $\phi:Y\ra X$ be a morphism between smooth and projective varieties.
    \begin{enumerate}[itemsep=0em]
        \item If $\phi$ is flat, then
            \begin{equation*}
                \phi_*(\Omega^*(Y)_{\fl_*\rmD})\subseteq\Omega^*(X)_{\fl_*\rmD}.
            \end{equation*}

        \item If $\phi$ is smooth, then
            \begin{equation*}
                \phi^*(\Omega^*(X)_{\fl_*\rmD})\subseteq \Omega^*(Y)_{\fl_*\rmD}.
            \end{equation*}
    \end{enumerate}
\end{prop}

\begin{proof}
    The first statement holds by definition. We will prove the second one in the following.

    Let $\alpha\in\Omega^*(X)_{\fl_*\rmD}$. Without loss of generality, we may assume that $\alpha = f_*\alpha'$ for some flat morphism $f:Z\ra X$ from a smooth projective variety $Z$, and $\alpha'\in\Omega^*(Z)_\rmD$. Since $\phi$ is smooth, the morphisms $f$ and $\phi$ are transverse. We consider the Cartesian square formed by $f$ and $\phi$.
    \begin{equation*}
        \begin{tikzcd}
            W \ar[r,"\phi'"] \ar[d,"f'"'] \ar[dr,phantom,"\square"] & Z \ar[d,"f"]\\
            Y \ar[r,"\phi"']& X
        \end{tikzcd}
    \end{equation*}
    By transverse commutativity, we have 
    \begin{equation}
    \label{eq: phi^*alpha=f'_*pfi'^*alpha'}
        \phi^*\alpha=\phi^*f_*\alpha'=f'_*\phi'^*\alpha'.
    \end{equation}

    Since $\alpha'\in\Omega^*(Z)_\rmD$, the class $\phi'^*\alpha'$ is in $\Omega^*(W)_\rmD$. Meanwhile, the morphism $f'$ is flat since $f$ is flat. Therefore, we obtain $\phi^*\alpha\in\Omega^*(Y)_{\fl_*\rmD}$ by (\ref{eq: phi^*alpha=f'_*pfi'^*alpha'}).
\end{proof}

\subsection{Chern classes revisited}
\label{subsection: Chern classes revisited}

We prove in this section that Chern classes in algebraic cobordism actually lie in subgroups $\Omega^*(X)_{\rmsm_*\rmD}$ defined before. We also present some technical lemmas that will be useful in the following sections.

\begin{prop}
\label{proposition: Chern classes belong to s_*D}
    Let $X$ be a smooth projective variety, and $E\ra X$ be a vector bundle over $X$. Then all Chern classes of $E$ belong to $\Omega^*(X)_{\rmsm_*\rmD}$.
\end{prop}

\begin{proof}
    All Segre classes belong to $\Omega^*(X)_{\rmsm_*\rmD}$ by definition. By Proposition \ref{proposition: L_*-subalgebras generated by Chern classes and Segre classes are the same}, every Chern class can be written as a polynomial of Segre classes with $\bbL^*$-coefficients. Therefore, they also belong to~$\Omega^*(X)_{\rmsm_*\rmD}$ since it is an $\bbL^*$-algebra by Proposition \ref{proposition: s_*D is a L_*-algebra and f_*D is a s_*D-module}.
\end{proof}

\begin{lem}
\label{lemma: c_r(E otimes L)}
    Let $X\in\Sm_k$. Let $E\ra X$ be a vector bundle of rank $r$, and $L\ra X$ be a line bundle. Then we have
    \begin{equation}
    \label{eq: c_r(E otimes L)}
        c_r(E\otimes L) = c_r(E)+\sum_{s\geq 1}Q_s(c_i(E),1\leq i \leq r)\cdot c_1(L)^s\ \text{in}\ \Omega^r(X)
    \end{equation}
   for some formal power series $Q_s(x_1,\cdots,x_r)\in\bbL^*[[x_1,\cdots,x_r]]$, which only depend on $r$ and $s$, not on~$X$,~$L$ or $E$.
\end{lem}

\begin{proof}
    Suppose that the Chern roots of $E$ are $\{\alpha_i,\ 1\leq i\leq r\}$, and $c_1(L)=\beta$. Then the Chern roots of $E\otimes L$ are given by $\{\alpha_i+_{\Omega}\beta, 1\leq i\leq r\}$. So we have 
    \begin{equation*}
        c_r(E\otimes L)=\prod_{i=1}^r(\alpha_i+_{\Omega}\beta)=\prod_{i=1}^r(\alpha_i+\beta+\sum_{j,k\geq 1}a_{j,k}\alpha_i^j\beta^k)=\prod_{i=1}^r\alpha_i + \sum_{s\geq 1}P_s(\alpha_i,1\leq i\leq r)\cdot\beta^s,
    \end{equation*}
    where $P_s(x_1,\cdots, x_r)\in\bbL^*[[x_1,\cdots, x_r]]$ are formal power series which only depend on $s$ and $r$. Since each $P_s$ is symmetric in its $r$ variables, it can be expressed as a formal power series in the elementary symmetric polynomials of these variables. Therefore, there exists formal power series~$Q_s(x_1,\cdots,x_r)\in\bbL^*[[x_1,\cdots,x_r]]$ such that 
    \begin{equation*}
        c_r(E\otimes L)=c_r(E)+\sum_{s\geq 1}Q_s(c_i(E),1\leq i \leq r)\cdot c_1(L)^s,
    \end{equation*}
    which finishes the proof.
\end{proof}

\begin{rem}
\label{remark: Q_s(E)}
    Since we have $\Omega^i(X)=\Omega_{\dim(X)-i}(X)=0$ when $i>\dim(X)$, the infinite sum on $s$ in the right hand side of equation (\ref{eq: c_r(E otimes L)}) is actually finite, and the formal power series $Q_s$ become polynomials when applied to Chern classes of $E$. The classes obtained in this way will be denoted by $Q_s(E)\in\Omega^*(X)$.
\end{rem}

Let $X\in\Sm_k$ and let $E\ra X$ be a vector bundle of rank $r$. Let $p:\bbP(E)\ra X$ be the projectivization of $E\ra X$. Denote by $t$ the first Chern class $c_1(O(1))$ of the tautological quotient line bundle $O(1)\ra\bbP(E)$. Then we have
\begin{equation*}
    p_*(t^{r-1})=1_X\in\CH^0(X).
\end{equation*}
However, this no longer holds in algebraic cobordism. Instead, we have the following result.

\begin{prop}
\label{proposition: pushforward from projective bundle}
    Let $X$ be a smooth projective variety, and let $E\ra X$ be a vector bundle of rank~$r$. Let $p: \bbP(E) \ra X$ be the projectivization of the bundle $E$. We denote by $O(1)\ra \bbP(E)$ the tautological quotient line bundle over $\bbP(E)$. Let $t=c_1(O(1))\in \Omega^1(\bbP(E))$. Then
    \begin{equation}
    \label{eq: 1_X=c_r(E)s_-r(E)+sum over s with t}
        1_X = c_r(E^\vee)\cdot s_{-r}(E^\vee) + \sum_{s\geq 1}Q_s(E^\vee)\cdot p_*(t^{s-1})
    \end{equation}
    in $\Omega^0(X)$, where $Q_s(E^\vee)$ are the polynomials with $\bbL^*$ coefficients in the Chern classes of $E^\vee$ that appear in Remark~\ref{remark: Q_s(E)}.
\end{prop}

\begin{proof}
    Denote by $q:\bbP(E\oplus O_X) \ra X$ the projectivization of the bundle $E\oplus O_X$ over $X$, and by~$O(1)'\ra \bbP(E\oplus O_X)$ the tautolgical quotient line bundle. The composition of the quotient morphism $q^*(E\oplus O_X)\ra O(1)'$ with the inclusion $q^*E\hra q^*(E\oplus O_X)$ defines a bundle morphism~$q^*E\ra O(1)'$, and hence a section of $O(1)'\otimes q^*E^\vee$ over $\bbP(E\oplus O_X)$. This section is transverse, with zero-locus given by $X\cong\bbP(O_X)\subseteq\bbP(E\oplus O_X)$. By Proposition~\ref{proposition: geometric interpretation of top Chern class} and Lemma~\ref{lemma: c_r(E otimes L)}, we obtain
    \begin{equation}
    \label{eq: [X hra P(E oplus O_X)]}
        [\bbP(O_X)\hra \bbP(E\oplus O_X)] = c_r(O(1)'\otimes q^*E^\vee) = c_r(q^*E^\vee)+\sum_{s\geq 1}Q_s(q^*E^\vee)\cdot t'^s,
    \end{equation}
    where $t'\coloneqq c_1(O(1)')\in\Omega^1(\bbP(E\oplus O_X))$.

    We pushforward equation (\ref{eq: [X hra P(E oplus O_X)]}) along $q$ and use the projection formula to obtain

    \begin{equation}
    \label{eq: temp}
        1_X = c_r(E^\vee)\cdot s_{-r}(E^\vee)+\sum_{s\geq 1}Q_s(E^\vee)\cdot q_*(t'^s).
    \end{equation}

    Here, we use the fact that $q_*(1)=s_{-r}(E^\vee)$ by the definition of Segre classes.

    We denote by $i:\bbP(E)\hra \bbP(E\oplus O_X)$ the morphism induced by the projection $E\oplus O_X\ra E$. Then $i^*O(1)'\cong O(1)$, and hence $i^*t' = t$. The composition of the inclusion $q^*O_X\hra q^*(E\oplus O_X)$ with the projection $q^*(E\oplus O_X)\ra O(1)'$ gives a transverse section of $O(1)'$ over $\bbP(E\oplus O_X)$, whose zero-locus can be identified with $\bbP(E)$. Therefore, we have $i_*(1)=t'$, and hence, for any $s\geq 1$,
    \begin{equation*}
        i_*(t^{s-1})=i_*(i^*t'^{s-1})= t'^{s-1}\cdot i_*(1)=t'^{s-1}\cdot t'=t'^s.
    \end{equation*}
    Since $p=q\circ i$, we have $q_*(t'^s)=q_*i_*(t^{s-1})=p_*(t^{s-1})$ for any $s\geq 1$. The equation (\ref{eq: temp}) becomes
    \begin{equation*}
        1_X = c_r(E^\vee)\cdot s_{-r}(E^\vee) + \sum_{s\geq 1}Q_s(E^\vee)\cdot p_*(t^{s-1}),
    \end{equation*}
    which finishes the proof.
\end{proof}

\subsection{Some stability results for $\Omega^*(X)_{\fl_*\rmD}$}
\label{subsection: some stability results for Omega_*(X)_{fl_*D}}
The goal of this section is to prove that $\Omega^*(X)_{\fl_*\rmD}$ is stable under the pushforwards along cbs blow-ups, which is Proposition \ref{proposition: f_*D is preserved under pushforward along cbs blow-up}. To establish this result, we start with the corresponding result for pushforwards along embeddings of smooth hypersurfaces.

\begin{prop}
\label{proposition: f_*D is preserved under pushforward of inclusion of smooth hypersurfaces}
    Let $X$ be a smooth projective variety, and let $j: Y \hra X$ be the inclusion of a smooth hypersurface $Y$ in $X$. Then 
    \begin{equation*}
        j_*(\Omega_*(Y)_{\fl_*\rmD}) \subseteq \Omega_*(X)_{\fl_*\rmD}.
    \end{equation*}
\end{prop}

\begin{proof}
    Suppose that $\dim(X)=n$. Denote by $i:\Gamma_j\hra Y\times X$ the inclusion of the graph of $j$. Then~$\Gamma_j\cong Y$ is of dimension $n-1$. We blow up $Y\times X$ along $\Gamma_j$ and obtain a smooth projective variety $T$. Let $\tau: T \ra Y\times X$ be the blow-up morphism, and let $E\subset T$ be the exceptional divisor. Denote by $\tau':E \ra \Gamma_j$ the restriction of $\tau$ to $E$. We have the following commutative diagram.

    \begin{equation*}
        \begin{tikzcd}
            E \ar[d,"\tau'"'] 
            \ar[r,hook,"i'"] & T=Bl_{\Gamma_j}(Y\times X) \ar[d,"\tau"] \\
            \Gamma_j \ar[r,hook, "i"'] & Y\times X
        \end{tikzcd}
    \end{equation*}

    Let $\alpha\in\Omega_*(Y)_{\fl_*\rmD}$. Then by Proposition \ref{proposition: formula of correspondence}, we have 
    \begin{equation}
    \label{eq: j_*alpha}
        j_*\alpha = pr_{X*}(pr_Y^*\alpha\cdot i_*(1_{\Gamma_j})),
    \end{equation}
    where $pr_X$ and $pr_Y$ are the projections from $Y\times X$ to $X$ and $Y$.

    Let $N\coloneqq N_{\Gamma_j/Y\times X}\ra \Gamma_j$ be the normal bundle. We can view $\tau': E \ra \Gamma_j$ as the projectivization of the bundle $\N^\vee$. We denote by $O(1)\ra \bbP(N^\vee)\cong E$ the tautological quotient line bundle over $E$, and define~$t\in\Omega^1(E)$ by $t\coloneqq c_1(O(1))$. By Proposition \ref{proposition: pushforward from projective bundle} applied with $E=N^\vee$, we obtain
    \begin{equation}
    \label{eq: 1_Gamma_j}
        1_{\Gamma_j} = c_n(N)\cdot s_{-n}(N) + \sum_{s\geq 1}Q_s(N)\cdot\tau'_*(t^{s-1}) = \sum_{s\geq 1}Q_s(N)\cdot\tau'_*(t^{s-1}),
    \end{equation}
    where $Q_s(N)$ are polynomials in Chern classes of $N$ with $\bbL^*$-coefficients as in Remark \ref{remark: Q_s(E)}. Here, we use the fact that $c_n(N)=0$ since it belongs to $\Omega^n(\Gamma_j)=\Omega_{-1}(\Gamma_j)=0$. 
    
    The normal bundle $N\ra\Gamma_j$ is isomorphic to the pullback of the tangent bundle $TX\ra X$ along the morphism $pr_X\circ i:\Gamma_j\hra Y\times X\ra X$. Therefore, we have $c_i(N)=i^*pr_X^*c_i(TX)$, and hence~$Q_s(N)=i^*pr_X^*Q_s(TX)$ for any $s\geq 1$. Applying this to (\ref{eq: 1_Gamma_j}), we obtain
    \begin{align}
        i_*(1_{\Gamma_j}) &= \sum_{s\geq 1} i_*(i^*pr_X^*Q_s(TX)\cdot\tau'_*(t^{s-1}))\notag\\
        &= \sum_{s\geq 1}pr_X^*Q_s(TX)\cdot\tau_*i'_*(t^{s-1}) \label{eq: i_*(1_Gamma_j)}.
    \end{align}
    Since $i'^*O_T(E)=O(-1)$, we have $t=c_1(O(1))=i'^*c_1(O_T(-E))$, and hence
    \begin{equation*}
        i'_*(t^{s-1})=i'_*(i'^*c_1(O_T(-E))^{s-1})=c_1(O_T(-E))^{s-1}\cdot i'_*(1)=[-E]_\Omega^{s-1}\cdot [E]_\Omega.
    \end{equation*}
    Therefore, equation (\ref{eq: i_*(1_Gamma_j)}) implies that
    \begin{equation}
    \label{eq: i_*(1_Gamma_j)'}
        i_*(1_{\Gamma_j})= \sum_{s\geq 1}pr_X^*Q_s(TX)\cdot\tau_*([-E]_\Omega^{s-1}\cdot [E]_\Omega).
    \end{equation}

Applying (\ref{eq: i_*(1_Gamma_j)'}) to (\ref{eq: j_*alpha}), we can write $j_*\alpha$ as 
    \begin{align*}
        j_*\alpha&=pr_{X*}(pr_Y^*\alpha\cdot i_*(1_{\Gamma_j}))\\
        &=\sum_{s\geq 1}pr_{X*}(pr_Y^*\alpha\cdot pr_X^*Q_s(TX)\cdot\tau_*([-E]_\Omega^{s-1}\cdot[E]_\Omega))\\
        &=\sum_{s\geq 1} Q_s(TX)\cdot pr_{X*}(pr_Y^*\alpha\cdot\tau_*([-E]_\Omega^{s-1}\cdot[E]_\Omega))\\
        &=\sum_{s\geq 1} Q_s(TX)\cdot q_*(p^*\alpha\cdot[-E]_\Omega^{s-1}\cdot[E]_\Omega),
    \end{align*}
    where $p=pr_Y\circ\tau:T\ra Y$ and $q=pr_X\circ\tau:T\ra X$. 
    
    The morphism $p$ is smooth and $q$ is flat by \cite[Lemma 3.8]{KollarVoisin2024}. Therefore, $p^*\alpha$ is in $\Omega_*(T)_{\fl_*\rmD}$ by Proposition \ref{proposition: stability results for f_*D}, and hence so is $p^*\alpha\cdot[-E]_\Omega^{s-1}\cdot[E]_\Omega$ by Remark \ref{remark: f_*D is a L_*-module and an Omega^*(X)_D-module}. Since $q$ is flat, its pushforward along $q$ belongs to $\Omega_*(X)_{\fl_*\rmD}$ by Proposition \ref{proposition: stability results for f_*D}. Meanwhile, by Proposition \ref{proposition: Chern classes belong to s_*D} and Proposition~\ref{proposition: s_*D is a L_*-algebra and f_*D is a s_*D-module}, the class $Q_s(TX)$ is in $\Omega_*(X)_{\rmsm_*\rmD}$. Therefore, the product $Q_s(TX)\cdot q_*(p^*\alpha\cdot[-E]_\Omega^{s-1}\cdot[E]_\Omega)$ belongs to $\Omega_*(X)_{\fl_*\rmD}$ by Proposition \ref{proposition: s_*D is a L_*-algebra and f_*D is a s_*D-module} for any $s\geq 1$. This implies that $j_*\alpha\in\Omega_*(X)_{\fl_*\rmD}$.
\end{proof}

\begin{prop}
\label{proposition: f_*D is preserved under pushforward of inclusion of smooth complete bundle-sections}
    Let $X$ be a smooth projective variety, and let $j: Y \hra X$ be the inclusion of a smooth complete bundle-section of $X$. Then 
    \begin{equation*}
        j_*(\Omega_*(Y)_{\fl_*\rmD}) \subseteq \Omega_*(X)_{\fl_*\rmD}
    \end{equation*}
\end{prop}

\begin{proof}
    Suppose that $Y$ is of codimension $r$ in $X$. We will prove the following statement:
    \begin{equation*}
        j_*(\Omega_d(Y)_{\fl_*\rmD}) \subseteq \Omega_d(X)_{\fl_*\rmD}
    \end{equation*}
    by double induction, first on $r$ and then on $d$. The case $r=1$ and arbitrary $d$ is exactly Proposition~\ref{proposition: f_*D is preserved under pushforward of inclusion of smooth hypersurfaces}, the case $d<0$ is trivial.

    Assume now that $r\geq 2$, $d\geq 0$, and that the statement holds when the codimension is less than~$r$ with arbitrary degree, and when the codimension is equal to $r$ with degree less than $d$.

    Suppose that $Y$ is defined as the zero-locus of a transverse section $\sigma$ of a vector bundle $E\ra X$ of rank $r$. We use diagram (\ref{diagram: cbs diagram}) again. In particular, we let the morphisms $p,j',j''$ be as in that diagram. Let $\alpha\in\Omega_d(Y)_{\fl_*\rmD}$. Without loss of generality, we may assume that~$\alpha=\phi_*(\beta)$ for some flat morphism $\phi:Z \ra X$ from a smooth projective variety $Z$, and some~$\beta\in\Omega_d(Z)_\rmD$.

    By Proposition \ref{proposition: pushforward from projective bundle} applied to the bundle $E\ra X$, we obtain 
    \begin{equation}
    \label{eq: 1_X=c_r(E)s_-r(E)+sum over s}
        1_X = c_r(E^\vee)\cdot s_{-r}(E^\vee) + \sum_{s\geq 1}Q_s(E^\vee)\cdot p_*(t^{s-1})\ \text{in}\ \Omega^0(X),
    \end{equation}
    where $t\coloneqq c_1(O(1))$ is the first Chern class of the tautological quotient line bundle $O(1)\ra\bbP(E)$, and $Q_s(E^\vee)$ are polynomials in Chern classes of $E^\vee$ with $\bbL^*$-coefficients as in Remark \ref{remark: Q_s(E)}.

    By (\ref{eq: 1_X=c_r(E)s_-r(E)+sum over s}) and the projection formula, we have 
    \begin{equation}
    \label{eq: j_*alpha=c_r(E)s_-r(E)j_*alpha+sum over s}
        j_*\alpha = 1_X\cdot j_*\alpha = c_r(E^\vee)\cdot s_{-r}(E^\vee)\cdot j_*\alpha + \sum_{s\geq 1}Q_s(E^\vee)\cdot p_*(t^{s-1}\cdot p^*j_*\alpha).
    \end{equation}
    
    Since $p$ is smooth, it is transverse to $j$, and we obtain
    \begin{equation}
    \label{eq: p^*j_*alpha}
        p^*j_*\alpha = j''_*j'_*(p|_Y)^*\alpha
    \end{equation}
    by transverse commutativity.

    As $\alpha\in\Omega_d(Y)_{\fl_*\rmD}$ and $p|_Y$ is smooth, the class $(p|_Y)^*\alpha$ lies in $\Omega_d(X'')_{\fl_*\rmD}$ by Proposition~\ref{proposition: stability results for f_*D}. Since $j':p^{-1}(Y)\hra X'$ is the inclusion of a smooth complete bundle-section of codimension $r-1$, and $j'':X'\hra\bbP(E)$ is the inclusion of a smooth hypersurface, equation (\ref{eq: p^*j_*alpha}) and the induction hypothesis imply that $p^*j_*\alpha\in\Omega_*(\bbP(E))_{\fl_*\rmD}$. By Remark~\ref{remark: f_*D is a L_*-module and an Omega^*(X)_D-module}, the class $t^{s-1}\cdot p^*j_*\alpha$ also lies in~$\Omega_*(\bbP(E))_{\fl_*\rmD}$ for any $s\geq 1$, hence we have $p_*(t^{s-1}\cdot p^*j_*\alpha)\in\Omega_*(X)_{\fl_*\rmD}$ by Proposition \ref{proposition: stability results for f_*D}. Meanwhile, by Proposition \ref{proposition: Chern classes belong to s_*D} and Proposition \ref{proposition: s_*D is a L_*-algebra and f_*D is a s_*D-module}, the class $Q_s(E^\vee)$ belongs to $\Omega_*(X)_{\rmsm_*\rmD}$ for all~$s \geq 1$. Therefore, the product $Q_s(E^\vee)\cdot p_*(t^{s-1}\cdot p^*j_*\alpha)$ lies in $\Omega_d(X)_{\fl_*\rmD}$ by Proposition \ref{proposition: s_*D is a L_*-algebra and f_*D is a s_*D-module}.

    Thus, to prove that $j_*\alpha\in\Omega_d(X)_{\fl_*\rmD}$, it suffices to show that the class $c_r(E^\vee) \cdot s_{-r}(E^\vee) \cdot j_*\alpha$ in equation (\ref{eq: j_*alpha=c_r(E)s_-r(E)j_*alpha+sum over s}) also lies in $\Omega_d(X)_{\fl_*\rmD}$.

    By Proposition \ref{proposition: L_*-subalgebras generated by Chern classes and Segre classes are the same}, we can write $s_{-r}(E^\vee)$ as 
    \begin{equation}
    \label{eq: s_-r(E)=sum of a_lambda c_lambda(E)}
        s_{-r}(E^\vee) = \sum_{\lambda}a_{\lambda}\cdot c_{\lambda}(E^\vee).
    \end{equation}
    Here, we sum over $\lambda\coloneqq(\lambda_1,\lambda_2,\ldots)$ with each $\lambda_i\in\bbZ_{\geq 0}$, where $c_\lambda(E^\vee)\coloneqq\prod_i c_i(E^\vee)^{\lambda_i}$, and~$a_\lambda\in\bbL_*$. By comparing degrees on both sides of (\ref{eq: s_-r(E)=sum of a_lambda c_lambda(E)}), we obtain $a_\lambda\in \bbL_{r+|\lambda|}$, where $|\lambda|\coloneqq\sum_i i\lambda_i$. In particular, we have $a_\lambda\in \bbL_{>0}$ for any $\lambda$. From (\ref{eq: s_-r(E)=sum of a_lambda c_lambda(E)}), we deduce that
    \begin{equation}
    \label{eq: c_r(E)s_-r(E)j_*alpha=sum over lambda}
        c_r(E^\vee)\cdot s_{-r}(E^\vee)\cdot j_*\alpha = \sum_\lambda a_\lambda\cdot j_*(c_r(j^*E^\vee)\cdot c_\lambda(j^*E^\vee)\cdot\alpha)
    \end{equation}
    in $\Omega_d(X)$. By Proposition \ref{proposition: Chern classes belong to s_*D} and Proposition~\ref{proposition: s_*D is a L_*-algebra and f_*D is a s_*D-module}, we have $c_r(j^*E^\vee)\cdot c_\lambda(j^*E^\vee)\cdot\alpha\in\Omega_*(Y)_{\fl_*\rmD}$. Since~$a_\lambda\in\bbL_{>0}$, the class $c_r(j^*E^\vee)\cdot c_\lambda(j^*E^\vee)\cdot\alpha$ lies in $\Omega_{<d}(Y)_{\fl_*\rmD}$. Therefore, by the induction hypothesis, we have $j_*(c_r(j^*E^\vee)\cdot c_\lambda(j^*E^\vee)\cdot\alpha)\in\Omega_{<d}(X)_{\fl_*\rmD}$ for any~$\lambda$. By Remark~\ref{remark: f_*D is a L_*-module and an Omega^*(X)_D-module}, it follows from (\ref{eq: c_r(E)s_-r(E)j_*alpha=sum over lambda}) that $c_r(E^\vee)\cdot s_{-r}(E^\vee)\cdot j_*\alpha\in\Omega_d(X)_{\fl_*\rmD}$, which completes the proof.
\end{proof}

\begin{prop}
\label{proposition: f_*D is preserved under pushforward along cbs blow-up}
    Let $X$ be a smooth projective variety, and let $Z\subset X$ be a smooth complete bundle-section in $X$. Denote by $\tau:\wt{X} \ra X$ the blow-up of $X$ along $Z$. Then 
    \begin{equation*}
        \tau_*(\Omega_*(\wt{X})_{\fl_*\rmD})\subseteq \Omega_*(X)_{\fl_*\rmD}.
    \end{equation*}
\end{prop}

\begin{proof}
    Since $Z$ is a smooth complete bundle-section, we may assume that $Z$ is the zero-locus of a transverse section $\sigma$ of a vector bundle $E\ra X$ over $X$. 

    Let $p:\bbP(E^\vee)\ra X$ be the projectivization of the dual bundle $E^\vee$ of $E$, and let $O(1)\ra\bbP(E^\vee)$ be the tautological quotient line bundle. Taking the dual of the surjection $p^*E^\vee\thra O(1)$, we obatain an injection $O(-1)\hra p^*E$. The section $\sigma$ induces a section $p^*\sigma$ of the bundle $p^*E\ra\bbP(E^\vee)$. We define $\bar{\sigma}$ to be the projection of this section to the quotient bundle $Q\coloneqq p^*E/O(-1)\ra\bbP(E^\vee)$. By a local computation, one can show that $\bar{\sigma}$ is a transverse section, and its zero-locus is isomorphic to the blow-up $\wt{X}$. So we have the following commutative diagram, where we use $j$ to denote the inclusion of the zero-locus of $\bar{\sigma}$ in $\bbP(E^\vee)$.

    \begin{equation}
    \label{diagram: cbs blow-up}
        \begin{tikzcd}
            Q=p^*E/O(-1) & p^*E \ar[d] \ar[l,two heads] & E \ar[d]\\
             & \bbP(E^\vee) \ar[r,"p"] \ar[u, bend right, "p^*\sigma"'] \ar[lu, "\bar{\sigma}"] & X \ar[u,bend right, "\sigma"']\\
             & \{\bar{\sigma}=0\}\cong\wt{X} \ar[u,hook,"j"] \ar[ur,"\tau"'] & 
        \end{tikzcd}
    \end{equation}

    As $\tau=p\circ j$, we have $\tau_* = p_* \circ j_*:\Omega_*(\wt{X})\ra\Omega_*(X)$. Since $p$ is flat and $j$ is the inclusion of a smooth complete bundle-section, the proof is completed by Proposition~\ref{proposition: stability results for f_*D} and Proposition~\ref{proposition: f_*D is preserved under pushforward of inclusion of smooth complete bundle-sections}.
\end{proof}

\section{Koll\'ar--Voisin cbs resolution and smoothability of cycles in $\Omega^{*}$}
\label{section: Koll\'ar--Voisin cbs resolution and smoothability of cycles in Omega^{*}}

\subsection{The Koll\'ar--Voisin cbs resolution theorem} 
\label{subsection: The Koll\'ar--Voisin cbs resolution theorem}

In this section, we use a result of Koll\'ar and Voisin on cbs resolutions (see \cite[Corollary 1.10]{KollarVoisin2024}) to prove Theorem \ref{theorem: all classes are f_*D}, which is an analog of~\cite[Theorem 1.6]{KollarVoisin2024} (see Theorem \ref{KV's theorem: CH(X)=CH(X)_flCh}) in the algebraic cobordism setting. We then improve this result to obtain Theorem~\ref{theorem: all classes are fl_*VAD} by showing that the subgroup $\Omega^*(X)_{\fl_*\rmVAD}\subseteq\Omega^*(X)$ is in fact an~$\bbL^*$-submodule (see Proposition \ref{proposition: f_*VAD is a L_*-module}), which implies that $\Omega^*(X)_{\fl_*\rmD}=\Omega^*(X)_{\fl_*\rmVAD}$. Although the identification of the corresponding subgroups in the Chow ring is obvious, it requires some work in our setting for the reasons explained in the introduction.

\begin{thm}
\label{theorem: all classes are f_*D}
    Let $X$ be a smooth projective variety. Then
    \begin{equation*}
        \Omega_*(X) = \Omega_*(X)_{\fl_*\rmD}.
    \end{equation*}
\end{thm}

\begin{proof}
    We prove that $\Omega_{\leq d}(X)=\Omega_{\leq d}(X)_{\fl_*\rmD}$ by induction on $d$. The case $d=-1$ holds trivially.

    Now we assume that $d\geq 0$, and that $\Omega_{<d}(X)=\Omega_{<d}(X)_{\fl_*\rmD}$ holds for any smooth projective variety $X$. We will show that $\Omega_d(X)=\Omega_d(X)_{\fl_*\rmD}$. Since $\Omega_d(X)$ is generated as a group by standard cycles (see Section \ref{subsection: Algebraic cobordism and its geometric interpretation}), i.e., classes of the form $[\phi:Z\ra X]$, where $Z$ is a smooth projective variety of dimension $d$, it suffices to show that every such class lies in $\Omega_d(X)_{\fl_*\rmD}$. 

    Since $Z$ is projective, we may take $N\in\bbZ$ with $N>4d$ such that $Z$ can be embedded in~$\bbP_k^N$. Let $i:Z\hra \bbP_k^N$ be the inclusion. Define $j\coloneqq(i,\phi):Z\ra \bbP_k^N\times X$. Let $pr_X:\bbP_k^N\times X \ra X$ be the projection. Then we have $\phi=pr_X\circ j$, and hence $[\phi:Z\ra X]=pr_{X*}[j:Z\ra\bbP_k^N\times X]$. Since $pr_X$ is flat, it suffices to show that $[j:Z\ra\bbP_k^N\times X]\in\Omega^*(\bbP_k^N\times X)_{\fl_*\rmD}$. Since $j$ is a closed embedding and $N>4d$, we may replace $X$ by $X\times\bbP_k^N$ and the class $[\phi:Z\ra X]$ by the class~$[j:Z\ra \bbP_k^N\times X]$. Therefore, we can assume that $\phi:Z \ra X$ is a closed embedding with $\dim(X)>4d$. 

    By the result of Koll\'ar and Voisin in \cite[Corollary 1.10]{KollarVoisin2024}, there exists a smooth cbs blow-up sequence, $i.e.$, a sequence of successive blow-ups along smooth complete bundle-sections
    \begin{equation*}
        \Pi: X_{r+1} \ra X_r \ra \cdots \ra X_0 \coloneqq X,
    \end{equation*}
    and a complete intersection subvariety $Z_{r+1}\subseteq X_{r+1}$ such that
    \begin{equation*}
        \Pi_*([Z_{r+1}])=[Z] \text{ in } \CH_d(X).
    \end{equation*}
    Moreover, one may assume that $Z_{r+1}$ is smooth, since it can be taken to be a general complete intersection of very ample divisors in a smooth divisor of $X_{r+1}$. This is important in our case, since it allows us to construct a standard cycle $[Z_{r+1}\hra X_{r+1}]\in\Omega_d(X_{r+1})$. Moreover, it belongs to~$\Omega_d(X_{r+1})_\rmD$ by construction.

    Let $\theta:\Omega\ra \CH$ be the natural morphism of oriented cohomology theories. It commutes with all pullbacks and projective pushforwards. In particular, it maps the standard cycles $[Z_{r+1}\hra X_{r+1}]$ and $[Z\hra X]$ to $[Z_{r+1}]\in \CH_d(X_{r+1})$ and $[Z]\in \CH_d(X)$ respectively. We have 
    \begin{align*}
        \theta(X)\circ \Pi_*([Z_{r+1}\hra X_{r+1}]) &= \Pi_*\circ \theta(X_{r+1})([Z_{r+1}\hra X_{r+1}])\\
        &= \Pi_*([Z_{r+1}]) = [Z] = \theta(X)([Z\hra X]).
    \end{align*}
    By Lemma~\ref{lemma: kernel of theta(X)}, we obtain
    \begin{equation}
    \label{eq: Pi_*([Z_r+1 hra X_r+1]) - [Z hra X] in ker(theta(X))}
        \Pi_*([Z_{r+1}\hra X_{r+1}]) - [Z\hra X] \in \ker(\theta(X))=\Omega_*(X)\times\bbL_{>0}\subseteq\Omega_*(X).
    \end{equation}

    The induction hypothesis and Remark \ref{remark: f_*D is a L_*-module and an Omega^*(X)_D-module} show that the degree $d$ part of $\Omega_*(X)\times\bbL_{>0}$ is contained in $\Omega_d(X)_{\fl_*\rmD}$. Therefore, (\ref{eq: Pi_*([Z_r+1 hra X_r+1]) - [Z hra X] in ker(theta(X))}) gives
    \begin{equation}
    \label{eq: Pi_*([Z_r+1 hra X_r+1]) - [Z hra X] in fl_*D}
        \Pi_*([Z_{r+1}\hra X_{r+1}])-[Z\hra X]\in\Omega_d(X)_{\fl_*\rmD}.
    \end{equation}
    Since $[Z_{r+1}\hra X_{r+1}]\in\Omega_d(X_{r+1})_\rmD\subseteq\Omega_d(X_{r+1})_{\fl_*\rmD}$, we obtain $\Pi_*([Z_{r+1}\hra X_{r+1}])\in\Omega_d(X)_{\fl_*\rmD}$ by Proposition~\ref{proposition: f_*D is preserved under pushforward along cbs blow-up}. Therefore, equation (\ref{eq: Pi_*([Z_r+1 hra X_r+1]) - [Z hra X] in fl_*D}) implies that $[Z\hra X]$ lies in $\Omega_d(X)_{\fl_*\rmD}$ as well, which completes the induction step.
\end{proof}

In fact, we can replace $\Omega^*(X)_{\fl_*\rmD}$ by $\Omega^*(X)_{\fl_*\rmVAD}$ in Theorem \ref{theorem: all classes are f_*D}, see Theorem \ref{theorem: all classes are fl_*VAD}. This relies on the following result, which shows that $\Omega^*(X)_{\fl_*\rmD}=\Omega^*(X)_{\fl_*\rmVAD}$.

\begin{prop}
\label{proposition: f_*VAD is a L_*-module}
    Let $X$ be a smooth projective variety. Then $\Omega_*(X)_{\fl_*\rmVAD}$ is an $\bbL_*$-module.
\end{prop}
\begin{proof}
    We consider the following statements $P(c,d)$ parametrized by $d\geq c\geq 0$:
    \begin{equation*}
        P(c,d):\ \bbL_{c}\times\Omega_{d-c}(X)_{\fl_*\rmVAD} \subseteq \Omega_d(X)_{\fl_*\rmVAD}.
    \end{equation*}
    Once we prove it for any $d\geq c\geq 0$, the original proposition is proved.

    In what follows, we prove these statements by double induction, first on $d$, then on $c$. The statement is trivial for $c=0$ and arbitrary $d$. So we may assume that the statement $P(c',d')$ is true for any pair $(c',d')$ satisfying the condition that (1) $d'<d$, or (2) $d'=d$ and $c'<c$. We want to prove that $P(c,d)$ also holds.

    Let $a\in\bbL_c$ and $\alpha\in\Omega_{d-c}(X)_{\fl_*\rmVAD}$. We aim to show that $a\times\alpha\in\Omega_d(X)_{\fl_*\rmVAD}$. We may assume that $a=[Z/k]$ is represented by some smooth projective variety $Z$ of dimension $c$, and that~$\alpha=\phi_*(\prod_{i=1}^sc_1(L_i))$ for some flat morphism $\phi:Y \ra X$ from a smooth projective variety $Y$, where $\{L_i\ra Y,\ 1\leq i \leq s\}$ are $s$ very ample line bundles over $Y$. By comparing the degrees, we have $\dim(Y)=d-c+s$.
    
    Let $\pi_Z:Z\ra \Spec(k)$ be the structure morphism of $Z$. We have $a=\pi_{Z*}(1_Z)$, and hence
    \begin{equation}
    \label{eq: a*alpha=(pi_Z*phi)_*()}
        a\times\alpha = \pi_{Z*}(1_Z)\times\phi_*(\prod_{i=1}^s c_1( L_i))
        = (\pi_Z\times\phi)_*(1_Z\times\prod_{i=1}^s c_1(L_i)).
    \end{equation}

    Let $p:Z\times Y\ra Y$ be the projection onto $Y$. Then we have $1_Z\times\beta = p^*\beta$ for any $\beta\in\Omega^*(Y)$. Therefore, equation (\ref{eq: a*alpha=(pi_Z*phi)_*()}) becomes 
    \begin{equation}
    \label{eq: a*alpha=(pi_Z*phi)_*(prod_i=1^s c_1(p^*L_i))}
        a\times\alpha = (\pi_Z\times\phi)_*p^*(\prod_{i=1}^s c_1(L_i)) = (\pi_Z\times\phi)_*(\prod_{i=1}^s c_1(p^*L_i)).
    \end{equation}
    
    Let $q:Z\times Y \ra Z$ be the projection onto $Z$. For any very ample line bundle $L\ra Z$ over~$Z$, the line bundle $q^*L\otimes p^*L_i\ra Z\times Y$ is very ample over $Z\times Y$ for any $1\leq i\leq s$. Since $\phi$ is flat, the morphism $\pi_Z\times\phi:Z\times Y\ra \Spec(k)\times X=X$ is also flat. Therefore, the class
    \begin{equation}
    \label{eq: definition of gamma}
        \gamma\coloneqq (\pi_Z\times\phi)_*(\prod_{i=1}^s c_1(q^*L\otimes p^*L_i))
    \end{equation}
    lies in $\Omega^*(X)_{\fl_*\rmVAD}$ by definition. We compute the difference between $\gamma$ and $a\times\alpha$.

    For each $1\leq i\leq s$, we have 
    \begin{equation}
    \label{eq: c_1(q^*L otimes p^*L_i) = c_1(q^*L) +_Omega c_1(p^*L_i)}
        c_1(q^*L\otimes p^*L_i) = c_1(q^*L) +_\Omega c_1(p^*L_i) = c_1(p^*L_i) + \sum_{j\geq 1}c_1(q^*L)^j\cdot P_j(c_1(p^*L_i))
    \end{equation}
    for some polynomials $P_j(x)\in \bbL^*[x]$ which depend only on $j$, and not depend on $L$ or $L_i$. By taking the product of (\ref{eq: c_1(q^*L otimes p^*L_i) = c_1(q^*L) +_Omega c_1(p^*L_i)}) over $1\leq i\leq s$, we obtain
    \begin{align}
        \prod_{i=1}^s c_1(q^*L\otimes p^*L_i)&= \prod_{i=1}^s c_1(p^*L_i) + \sum_{j\geq 1}c_1(q^*L)^j \cdot Q_j(c_1(p^*L_i), 1\leq i\leq s) \notag \\
        &= \prod_{i=1}^s c_1(p^*L_i) + \sum_{j\geq 1}c_1(L)^j \times Q_j(c_1(L_i), 1\leq i\leq s) \label{eq: prod_{i=1}^s c_1(q^*L otimes p^*L_i)}
    \end{align}
    for some polynomials $Q_j(x_1,\cdots, x_s)\in\bbL^*[x_1,\cdots,x_s]$ which depend only on $s$, and not on $L$ or $L_i$.

    Pushing forward (\ref{eq: prod_{i=1}^s c_1(q^*L otimes p^*L_i)}) along $\pi_{Z}\times\phi$ and using (\ref{eq: a*alpha=(pi_Z*phi)_*(prod_i=1^s c_1(p^*L_i))}) and (\ref{eq: definition of gamma}), we obtain 
    \begin{equation}
    \label{eq: gamma=a*alpha+sum over j}
        \gamma = a\times\alpha +\sum_{j\geq 1}\pi_{Z*}(c_1(L)^j)\times \phi_*Q_j(c_1(L_i),1\leq i\leq s)\ \text{in}\ \Omega_d(X).
    \end{equation}
    Since $c_1(L)^j\in\Omega_{c-j}(Z)$ and hence vanishes for $j>c$, the sum in (\ref{eq: gamma=a*alpha+sum over j}) is in fact finite, with $1\leq j\leq c$.

    Denote by $Q_j(L_\bullet)$ the class $Q_j(c_1(L_i),1\leq i\leq s)$, it depends only on $L_i,1\leq i\leq s$ and not on $L$. For $1\leq j<c$, since $Q_j$ is a polynomial with~$\bbL^*$ coefficients and the~$L_i$ are very ample line bundles over $X$, each class $\phi_*Q_j(L_\bullet)$ in (\ref{eq: gamma=a*alpha+sum over j}) is in~$\bbL_*\times\Omega_*(X)_{\fl_*\rmVAD}$. Meanwhile, since $\pi_{Z*}(c_1(L)^j)\in\bbL_{c-j}$, the degree of the class $\phi_*Q_j(L_\bullet)$ is equal to $d-c+j$, which is strictly lower than $d$ since $j<c$. Therefore, by induction hypothesis (1), the class $\phi_*Q_j(L_\bullet)$ is in $\Omega_{d-c+j}(X)_{\fl_*\rmVAD}$. Since $j\geq 1$, we may apply induction hypothesis (2) to the coefficient $\pi_{Z*}(c_1(L)^j)$ in $\bbL_{c-j}$ and the class $\phi_*Q_j(L_\bullet)$ in~$\Omega_{d-c+j}(X)_{\fl_*\rmVAD}$ to show that their product lies in $\Omega_d(X)_{\fl_*\rmVAD}$.

    We denote by $\beta(L)$ the remaining term in the right hand side of (\ref{eq: gamma=a*alpha+sum over j}) with index $j=c$, i.e., 
    \begin{equation}
    \label{eq: definition of beta(L)}
        \beta(L)\coloneqq\pi_{Z*}(c_1(L)^c)\times\phi_*Q_c(L_\bullet).
    \end{equation}

    The argument above shows that 
    \begin{equation}
    \label{eq: a*alpha+beta(L)}
        a\times\alpha +\beta(L)\in \Omega_d(X)_{\fl_*\rmVAD}.
    \end{equation}
    Since this result holds for any very ample line bundle $L$ over $X$, we may replace $L$ in (\ref{eq: a*alpha+beta(L)}) by $L^{\otimes m}$ for any positive integer $m$, and obtain 
    \begin{equation*}
        a\times\alpha + \beta(L^{\otimes m})\in\Omega_d(X)_{\fl_*\rmVAD}.
    \end{equation*}
    
    We now compute $\beta(L^{\otimes m})$. Since we have
    \begin{equation*}
        c_1(L^{\otimes m}) = mc_1(L)+ \sum_{k\geq 2}a_k\times c_1(L)^k
    \end{equation*}
    for some $a_k\in\bbL_*$ and 
    \begin{equation*}
        c_1(L)^{c+1}=0\in\Omega_{-1}(Z),
    \end{equation*}
    it follows that 
    \begin{equation*}
        c_1(L^{\otimes m})^c=m^c c_1(L)^c\ \text{in}\ \Omega_0(Z).
    \end{equation*}
    Since the class $Q_c(L_\bullet)$ depends only on $L_i,1\leq i\leq s$ and not on $L$, equation (\ref{eq: definition of beta(L)}) implies that 
    \begin{equation*}
        \beta(L^{\otimes m})=m^c\beta(L),
    \end{equation*}
    and hence 
    \begin{equation}
    \label{eq: a*alpha+m^c beta(L) in fl_*VAD}
        a\times\alpha + m^c\beta(L)\in\Omega_d(X)_{\fl_*\rmVAD},
    \end{equation}
    for any positive integer $m$.

    Replacing $m$ by another positive integer $m'$ in (\ref{eq: a*alpha+m^c beta(L) in fl_*VAD}) and subtracting the two classes, we obtain 
    \begin{equation*}
    \label{eq: (m^c-m'^c)beta(L) in fl_*VAD}
        (m^c-m'^c)\beta(L)\in\Omega_d(X)_{\fl_*\rmVAD}.
    \end{equation*}
    Choosing another pair $(l,l')$ such that $\gcd(m^c-m'^c,l^c-l'^c)=1$ and applying B\'ezout's theorem, we deduce that 
    \begin{equation}
    \label{eq: beta(L) in fl_*VAD}
        \beta(L)\in\Omega_d(X)_{\fl_*\rmVAD}.
    \end{equation}
    Combining (\ref{eq: a*alpha+beta(L)}) with (\ref{eq: beta(L) in fl_*VAD}), we obtain 
    \begin{equation*}
        a\times\alpha\in\Omega_d(X)_{\fl_*\rmVAD},
    \end{equation*}
    which completes the induction step.
\end{proof}

\begin{thm}
\label{theorem: all classes are fl_*VAD}
    Let $X$ be a smooth projective variety. Then 
    \begin{equation*}
        \Omega^*(X) = \Omega^*(X)_{\fl_*\rmD} = \Omega^*(X)_{\fl_*\rmVAD}.
    \end{equation*}
\end{thm}
\begin{proof}
    Since any line bundle $L\ra X$ over $X$ can be written as $L=L_1\otimes L_2^\vee$ where $L_1,L_2\ra X$ are very ample line bundles, the class $c_1(L)=c_1(L_1)-_\Omega c_1(L_2)$ can be written as a polynomial in the classes $c_1(L_i)$ with $\bbL^*$-coefficients. Therefore, we have 
    \begin{equation*}
        \Omega^*(X)_\rmD\subseteq \bbL^*\times\Omega^*(X)_{\rmVAD}\subseteq\Omega^*(X),\ \text{for all}\ X\in\Sm_k,
    \end{equation*}
    and hence 
    \begin{equation*}
        \Omega^*(X)_{\fl_*\rmD}\subseteq \bbL^*\times\Omega^*(X)_{\fl_*\rmVAD}=\Omega^*(X)_{\fl_*\rmVAD},\ \text{for all}\ X\in\Sm_k,
    \end{equation*}
    by Proposition \ref{proposition: f_*VAD is a L_*-module}. Meanwhile, the inclusion $\Omega^*(X)_{\fl_*\rmVAD}\subseteq\Omega^*(X)_{\fl_*\rmD}$ holds by definition. Therefore, we have
    \begin{equation*}
        \Omega^*(X)_{\fl_*\rmD}=\Omega^*(X)_{\fl_*\rmVAD}.
    \end{equation*}
    Finally, we complete the proof by applying Theorem \ref{theorem: all classes are f_*D}.
\end{proof}

\subsection{Smoothing cycles in $\Omega^*$}

In this section, we will prove our main result, Theorem \ref{main theorem}, by combining Theorem \ref{theorem: all classes are fl_*VAD} with the following proposition.

\begin{prop}
\label{proposition: f_*VAD is smoothable within the Whitney range}
    Let $X$ be a smooth projective variety. Then
    \begin{equation*}
        \Omega_d(X)_{\fl_*\rmVAD} \subseteq \Omega_d(X)_{\rmsm}
    \end{equation*}
    for $2d<\dim(X)$.   
\end{prop}

\begin{proof}
    Let $\alpha\in\Omega_d(X)_{\fl_*\rmVAD}$. We may assume that $\alpha=\phi_*(\prod_{i=1}^s c_1(L_i))$, where $\phi:Y\ra X$ is a flat morphism from a smooth projective variety $Y$ to $X$, the line bundles $L_i\ra Y$ are very ample, and $s=\dim(Y)-d$.

    Since the $L_i$ are very ample line bundles, general divisors $D_i$ in the linear systems $|L_i|$ have smooth complete intersection $Z\coloneqq\bigcap_{i=1}^s D_i$ in $Y$. Moreover, we may further require that $\bigcap_{i=1}^j D_i$ is a smooth hypersurface in $\bigcap_{i=1}^{j-1} D_i$ for any $1\leq j\leq s$. This implies that 
    \begin{equation*}
        \prod_{i=1}^s c_1(L_i)=[Z\hra Y]\ \text{in}\ \Omega_d(Y).
    \end{equation*} 

    By \cite[Remark 2.2]{KollarVoisin2024}, we may apply \cite[Proposition 2.1]{KollarVoisin2024} to the morphism $\phi$ and the subvariety $Z$. It follows that $\phi|_Z:Z\ra \phi(Z)$ is an isomorphism, and hence the class
    \begin{equation*}
        \alpha=\phi_*(\prod_{i=1}^s c_1(L_i))=\phi_*([Z\hra Y])=[\phi(Z)\hra X]
    \end{equation*}
    belongs to $\Omega_d(X)_{\sm}$.
\end{proof}

Combining Proposition \ref{proposition: f_*VAD is smoothable within the Whitney range} with Theorem \ref{theorem: all classes are fl_*VAD}, we obtain our main result.

\begin{thm}
\label{main theorem}
    Let $X$ be a smooth projective variety. Then all classes in $\Omega_d(X)$ are smoothable in the Whitney range $2d<\dim(X)$, i.e., we have
    \begin{equation*}
        \Omega_d(X) = \Omega_d(X)_{\sm}
    \end{equation*}
    for $2d<\dim(X)$.
\end{thm}

\bibliographystyle{myamsalpha}
\bibliography{references}

\end{document}